\pdfoutput=1
\documentclass{amsart}
\usepackage{booktabs}
\usepackage{amssymb}
\makeindex

\usepackage{hyperref}
\hypersetup{
    colorlinks=true,       % false: boxed links; true: colored links
    linkcolor=blue,          % color of internal links
    citecolor=blue,        % color of links to bibliography
    filecolor=blue,      %t colo r of file links
    urlcolor=blue           % color of external links
}
\usepackage[all]{xy}
\usepackage{tikz}
\usepackage{todonotes}
\usepackage[alphabetic,backrefs,lite]{amsrefs}
\usepackage{enumerate}
\usepackage{verbatim}
\usepackage{multirow}
\usepackage{pgf,tikz}
\usetikzlibrary{arrows}

\CompileMatrices

\newtheorem{theorem}{Theorem}[section]
\newtheorem{lemma}[theorem]{Lemma}

\newtheorem{proposition}[theorem]{Proposition}
\newtheorem{corollary}[theorem]{Corollary}
\theoremstyle{definition}
\newtheorem{definition}[theorem]{Definition}
\newtheorem{example}[theorem]{Example}

\newtheorem{remark}[theorem]{Remark}

\def\cc{{\mathbb C}}

\def\zz{{\mathbb Z}}

\def\pp{{\mathbb P}}

\def\Osh{{\mathcal O}}
\def\R{{\mathcal R}}

\def\Sym{\operatorname{Sym}}
\def\Div{\operatorname{Div}}

\def\Cl{\operatorname{Cl}}
\def\Pic{\operatorname{Pic}}

\def\rk{\mbox{rank}}
\renewcommand{\div}{\operatorname{div}}

\begin{document}

\title{Cox rings of extremal rational elliptic surfaces}

\author{Michela Artebani}
\address{
Departamento de Matem\'atica, \newline
Universidad de Concepci\'on, \newline
Casilla 160-C,
Concepci\'on, Chile}
\email{martebani@udec.cl}

\author{Alice Garbagnati}
\address{Dipartimento di Matematica\newline
Universit\`a Statale degli Studi di Milano\newline
via Saldini, 50, I-20133 Milano, Italy}
\email{alice.garbagnati@unimi.it} 

\author{Antonio Laface}
\address{
Departamento de Matem\'atica, \newline
Universidad de Concepci\'on, \newline
Casilla 160-C,
Concepci\'on, Chile}
\email{alaface@udec.cl}

\thanks{The first author has been partially 
supported by Proyecto FONDECYT Regular N. 1110249.
The third author has been partially supported 
by Proyecto FONDECYT Regular N. 1110096}

\begin{abstract}
In this paper we determine a minimal set of generators 
for the Cox rings of extremal rational elliptic surfaces.
Moreover, we develop a technique for computing the 
ideal of relations between them which allows, in all but three 
cases, to provide a presentation of the Cox ring.
\end{abstract}
\maketitle

\section*{Introduction}
Let $X$ be a smooth projective 
rational surface 
over $\cc$ and $\pi: X\to\pp^1$ 
be a jacobian elliptic fibration on $X$.
%which admits a section,
%that is a morphism 
%$\sigma:\pp^1\to X$ such that
%$\pi\circ\sigma=\id$.
In this paper we are interested 
in the Cox ring of  $X$,
which is the $\Pic(X)$-graded ring 
defined as follows (see section 1):
\[
\R(X):=\bigoplus _{[D]\in \Pic(X)}H^0(X,\mathcal O_X(D)).
\]
By~\cite{ArLa} the Cox ring
$\R(X)$ is a finitely generated 
algebra if and only if the Mordell-Weil
group of $\pi$ is finite, which means that $\pi$ admits
just a finite number of 
sections. The elliptic fibrations
with such property 
are called {\em extremal}, after
the work of Miranda and Persson
who classified them in~\cite{MiPe}.
The following table describes 
the Kodaira type of the singular fibers  
and the structure of the Mordell-Weil 
group for each such surface.

\begin{table}[h]
\begin{center}
\begin{tabular}{l|l|l||l|l|l}
 Surface & Type &  ${\rm MW}(\pi)$ & Surface & Type & ${\rm MW}(\pi)$\\
 \midrule
 $X_{22}$ & ${\rm II}^*\, {\rm II}$ & $\{0\}$ 
 & $X_{431}$ & ${\rm IV}^*\, {\rm I_3}\, {\rm I_1}$ & $\zz/3\zz$\\
 $X_{211}$ & ${\rm II}^*\, {\rm 2I_1}$ & $\{0\}$ 
 & $X_{222}$ & ${\rm I}_2^*\, {\rm 2I_2}$ & $(\zz/2\zz)^{\oplus 2}$\\
 $X_{411}$ & ${\rm I}_4^*\, {\rm 2I_1}$ & $\zz/2\zz$ 
 & $X_{141}$ & ${\rm I}_1^*\, {\rm I_4}\, {\rm I_1}$ & $\zz/4\zz$\\
 $X_{9111}$ & ${\rm I_9}\, {\rm 3I_1}$ & $\zz/3\zz$ 
 & $X_{6321}$ & ${\rm I_6}\, {\rm I_3}\, {\rm I_2}\, {\rm I_1}$ & $\zz/6\zz$\\
 $X_{33}$ & ${\rm III}^*\, {\rm III}$ & $\zz/2\zz$ 
 & $X_{11}(a)$ & ${\rm 2I}_0^*$ & $(\zz/2\zz)^{\oplus 2}$ \\
 $X_{321}$ & ${\rm III}^*\, {\rm I_2}\, {\rm I_1}$ & $\zz/2\zz$ 
 & $X_{5511}$ & ${\rm 2I_5}\, {\rm 2I_1}$ & $\zz/5\zz$\\
 $X_{8211}$ & ${\rm I_8}\, {\rm I_2}\, {\rm 2I_1}$ & $\zz/4\zz$ 
 & $X_{4422}$ & ${\rm 2I_4}\, {\rm 2I_2}$ & $\zz/4\zz\oplus \zz/2\zz$\\
 $X_{44}$ & ${\rm IV}^*\, {\rm IV}$ & $\zz/3\zz$ 
 & $X_{3333}$ & ${\rm 4I_3}$ & $(\zz/3\zz)^{\oplus 2}$
\end{tabular}
\end{center}
\vspace{3mm}
\caption{Extremal rational elliptic surfaces.}\label{extremal}
\end{table}
\vspace{-0.3cm}
The aim of this paper is to give an explicit presentation 
for the Cox ring of extremal elliptic surfaces, i.e. to provide 
a minimal set of generators for $\R(X)$ and to describe the ideal 
of relations between them.

We now give a short description of the content of the paper.

In the first section we recall the basic properties of 
elliptic fibrations and some preliminary 
results about Cox rings.

In the second section we provide a set of 
generators for the Cox ring 
of any extremal rational elliptic surface.
The elements of such set are the distinguished sections 
introduced in Definition \ref{dis}.
%Such generators are defining sections 
%of: $(-1)$-curves, $(-2)$-curves, 
%smooth fibers of conic bundles having a unique 
%reducible fiber, a smooth fiber of $\pi$ 
%if it has a unique 
%reducible fiber and finally pull-backs 
%of lines through morphisms $X\to \pp^2$ 
%which contract $9$ irreducible curves to one point.
%{\bf a me sembra che nel caso $X_{911}$ le 3-sezioni che generano l'anello di Cox non siano pull back di rette, ma di una curva razionale di grado 4!}

The third section deals with elliptic surfaces 
of complexity one, i.e. carrying 
an action of $\cc^*$. 
These are the surfaces $X_{22}, X_{33}, X_{44}$ 
and $X_{11}(a), a\in \cc-\{0\}$.
The Cox ring of such surfaces 
is computed by means of a technique 
developed in~\cite[\S 3.3]{Hau}.

In the fourth section we provide a 
method for computing the 
ideal of relations of $\R(X)$.
This allows to compute explicitely 
the Cox ring of all extremal rational elliptic 
surfaces except for $X_{8211}, X_{9111}, X_{4422}$.

Finally, in section five we apply the previous results 
to compute the Cox ring of certain generalized del Pezzo surfaces 
which are dominated by extremal rational elliptic surfaces.
\vspace{0.2cm}

\noindent {\em Aknowledgments.} We thank Damiano Testa for several useful discussions.

\section{Basic setup}
\subsection{Elliptic surfaces}
We briefly recall some well known facts about elliptic fibrations, 
a good and recent reference for these results is \cite{SS}.

Let $X$ be a smooth projective surface over $\cc$ and 
$\pi: X\to\pp^1$ be an  {\em elliptic fibration},
that is a morphism whose general fiber  
is a smooth curve of genus one.  
In what follows we will assume $\pi$ to be jacobian, 
i.e. $\pi$ admits a section.
Given a zero section $P_0$, the set of sections 
of  $\pi$ forms a finitely generated abelian group, 
the {\em Mordell-Weil group} of $\pi$, which will be 
denoted by $MW(\mathcal{\pi})$.
%This will be called the zero section in what follows. 
%Fixed one section as a zero, the set of sections 
%is a finitely generated abelian group,
%the {\em Mordell-Weil group} of $\pi$.
%Every elliptic fibration can be regarded as an elliptic curve over the function field of the basis. We will assume $D\simeq \mathbb{P}^1$ with homogeneous coordinate $(\tau:\sigma)$. Under this assumption each elliptic fibration admits a minimal Weierstrass equation of the form
%$$y^2=x^3+A(\tau,\sigma)x+B(\tau,\sigma),\ \ \ A(\tau,\sigma),\ B(\tau,\sigma)\in \C[\tau,\sigma]_{hom},\ \ \deg A(\tau,\sigma)=4m,\ \deg B(\tau,\sigma)=6m
%$$
%for a certain $m\in \N_{>0}$ and where there exists no polynomials $C(\tau,\sigma)$ such that $C(\tau,\sigma)^4|A(\tau,\sigma)$ and $C(\tau,\sigma)^6|B(\tau,\sigma)$. If $S$ is a K3 surface, then $m=2$, i.e. $\deg(A(\tau,\sigma))=8$, $\deg(B(\tau,\sigma))=12$.
%We use the notation $A(\tau)$ and $B(\tau)$ to indicate the polynomials $A(\tau,1)$ and $B(\tau,1)$.\\
%There are a finite number of singular fibers, which are the fibers over the points $\overline{\tau}\in\mathbb{P}^1$ where $\Delta(\overline\tau)=-16(4A^3(\overline\tau)+27B^2(\overline\tau))$ is zero. For each singular fiber $F_{\overline{\tau}}$ of the fibration we will denote by $\delta(F_{\overline{\tau}})$ the multiplicity of zero of $\Delta$ in $\overline{\tau}$.

The singular fibers of an elliptic fibration  
have been classified by Kodaira \cite{K}
according to the types $
{\rm I}_n\ (n\geq 1),  {\rm II}, {\rm III}, {\rm IV}, {\rm II}^*, 
{\rm III}^*, {\rm IV}^*, {\rm I}^*_n\ (n\geq 0).$
The intersection graph of some of the reducible fibers is 
represented in Figure \ref{sing},
where each vertex represents a $(-2)$-curve and 
the number near to it is its multiplicity.
 \begin{figure}[h!!]
 \definecolor{qqqqff}{rgb}{0,0,1}
\definecolor{cqcqcq}{rgb}{1,1,1}
\begin{tikzpicture}[line cap=round,line join=round,>=triangle 45,x=1.0cm,y=1.0cm, scale=0.54]
\draw [color=cqcqcq,dash pattern=on 3pt off 3pt, xstep=2.0cm,ystep=2.0cm] (-7,-8) grid (16,4);
\clip(-7,-8) rectangle (16,4);
\draw (-0.5,2.7) node[anchor=north west] {\tiny{$ \Theta_7$}};
\draw (-0.3,2) node[anchor=north west] {\tiny{$2$}};
\draw (0.6,2.7) node[anchor=north west] {\tiny{$ \Theta_6 $}};
\draw (0.7,2) node[anchor=north west] {\tiny{$4$}};
\draw (1.6,2.7) node[anchor=north west] {\tiny{$ \Theta_5 $}};
\draw (1.4,2) node[anchor=north west] {\tiny{$6$}};
\draw (2.6,2.7) node[anchor=north west] {\tiny{$ \Theta_4 $}};
\draw (2.7,2) node[anchor=north west] {\tiny{$5$}};
\draw (3.6,2.7) node[anchor=north west] {\tiny{$ \Theta_3 $}};
\draw (3.7,2) node[anchor=north west] {\tiny{$4$}};
\draw (4.6,2.7) node[anchor=north west] {\tiny{$ \Theta_2 $}};
\draw (4.7,2.) node[anchor=north west] {\tiny{$3$}};
\draw (5.6,2.7) node[anchor=north west] {\tiny{$ \Theta_1 $}};
\draw (5.7,2) node[anchor=north west] {\tiny{$2$}};
\draw (6.7,2.7) node[anchor=north west] {\tiny{$ \Theta_0 $}};
\draw (6.7,2) node[anchor=north west] {\tiny{$1$}};
\draw (2,1.4) node[anchor=north west] {\tiny{$ \Theta_8 $}};
\draw (1.4,1.4) node[anchor=north west] {\tiny{$3$}};
\foreach \x in {0,1,2,3,4,5,6}
 { \draw [line width=1pt] (\x.2,2)-- (\x.8,2); }
\draw [line width=1pt] (2,1.8)-- (2,1.2);
\draw (-5,2.7) node[anchor=north west] {$ {\rm II}^*\ (E_8)$};
\foreach \x in {0,1,2,3,4,5}
 { \draw [line width=1pt] (\x.2,0)-- (\x.8,0); }
\draw [line width=1pt] (3,-.2)-- (3,-.8);
\draw (-0.4,0.7) node[anchor=north west] {\tiny{$\Theta_0$}};
\draw (-0.3,0) node[anchor=north west] {\tiny{$1$}};
\draw (0.6,0.7) node[anchor=north west] {\tiny{$\Theta_1$}};
\draw (0.7,0) node[anchor=north west] {\tiny{$2$}};
\draw (1.6,0.7) node[anchor=north west] {\tiny{$\Theta_2$}};
\draw (1.7,0) node[anchor=north west] {\tiny{$3$}};
\draw (2.6,0.7) node[anchor=north west] {\tiny{$\Theta_3$}};
\draw (2.4,0) node[anchor=north west] {\tiny{$4$}};
\draw (3.6,0.7) node[anchor=north west] {\tiny{$\Theta_4$}};
\draw (3.7,0) node[anchor=north west] {\tiny{$3$}};
\draw (4.6,0.7) node[anchor=north west] {\tiny{$\Theta_5$}};
\draw (4.7,0) node[anchor=north west] {\tiny{$2$}};
\draw (5.6,0.7) node[anchor=north west] {\tiny{$\Theta_6$}};
\draw (5.7,0) node[anchor=north west] {\tiny{$1$}};
\draw (3,-0.6) node[anchor=north west] {\tiny{$\Theta_7$}};
\draw (2.4,-0.6) node[anchor=north west] {\tiny{$2$}};
\draw (-5,0.7) node[anchor=north west] {${\rm III}^*\ (E_7)$};
\foreach \x in {0,1,2,3}
 { \draw [line width=1pt] (\x.2,-2)-- (\x.8,-2); }
\foreach \x in {-3,-2}
 { \draw [line width=1pt] (2,\x.2)-- (2,\x.8); }
%\draw [line width=1pt] (2,-2)-- (2,-4);
\draw (-0.4,-1.3) node[anchor=north west] {\tiny{$ \Theta_0 $}};
\draw (-0.3,-2) node[anchor=north west] {\tiny{$1$}};
\draw (0.6,-1.3) node[anchor=north west] {\tiny{$ \Theta_1 $}};
\draw (0.7,-2) node[anchor=north west] {\tiny{$ 2 $}};
\draw (1.6,-1.3) node[anchor=north west] {\tiny{$ \Theta_2 $}};
\draw (1.4,-2) node[anchor=north west] {\tiny{$3$}};
\draw (2.6,-1.3) node[anchor=north west] {\tiny{$ \Theta_3 $}};
\draw (2.7,-2) node[anchor=north west] {\tiny{$ 2$}};
\draw (3.6,-1.3) node[anchor=north west] {\tiny{$ \Theta_4 $}};
\draw (3.7,-2) node[anchor=north west] {\tiny{$1$}};
\draw (2,-2.7) node[anchor=north west] {\tiny{$\Theta_5$}};
\draw (1.4,-2.7) node[anchor=north west] {\tiny{$2$}};
\draw (2,-3.6) node[anchor=north west] {\tiny{$ \Theta_6 $}};
\draw (1.4,-3.6) node[anchor=north west] {\tiny{$1$}};
\draw (-5,-1.2) node[anchor=north west] {${\rm IV}^*\ (E_6)$};
\draw (-2,-4.6) node[anchor=north west] {\tiny{$\Theta_0$}};
\draw (-1,-4.6) node[anchor=north west] {\tiny{$1$}};
\draw (-1.1,-5.6) node[anchor=north west] {\tiny{$\Theta_2$}};
\draw (-0.3,-6) node[anchor=north west] {\tiny{$2$}};
\draw (0.6,-5.3) node[anchor=north west]{};
\draw (0.7,-6) node[anchor=north west] {\tiny{$2$}};
\draw (2.2,-5.3) node[anchor=north west] {};
\draw (2.7,-6) node[anchor=north west] {\tiny{$2$}};
\draw (4.1,-5.6) node[anchor=north west] {\tiny{$\Theta_{n+2}$}};
\draw (3.7,-6) node[anchor=north west] {\tiny{$2$}};
\draw (5,-4.6) node[anchor=north west] {\tiny{$\Theta_{n+3}$}};
\draw (4.4,-4.6) node[anchor=north west] {\tiny{$1$}};
\draw (-5.2,-4) node[anchor=north west] {${\rm I}^*_n\ (D_{n+4})$};
\draw [line width=1pt] (-.8,-5.2)-- (-0.2,-5.8);
\draw [line width=1pt] (0.2,-6)-- (.8,-6);
\draw [line width=1pt] (3.2,-6)-- (3.8,-6);
\draw [line width=1pt] (4.2,-5.8)-- (4.8,-5.2);
\draw [line width=1pt] (4.2,-6.2)-- (4.8,-6.8);
\draw [line width=1pt] (-.8,-6.8)-- (-.2,-6.2);
\draw [line width=1pt,dash pattern=on 5pt off 5pt] (1.2,-6)-- (2.8,-6);
\draw (-1.9,-6.7) node[anchor=north west] {\tiny{$\Theta_1$}};
\draw (-1,-6.7) node[anchor=north west] {\tiny{$1$}};
\draw (5,-6.7) node[anchor=north west] {\tiny{$\Theta_{n+4}$}};
\draw (4.4,-6.7) node[anchor=north west] {\tiny{$1$}};
\draw [line width=1pt] (7.2,-2)-- (7.8,-2);
\draw [line width=1pt] (7,-2.2)-- (7,-3.8);
\draw [line width=1pt] (7.2,-4)-- (7.8,-4);
\draw [line width=1pt] (10,-2.2)-- (10,-3.8);
\draw [line width=1pt] (7,-3.8)-- (7,-2.2);
\draw [line width=1pt,dash pattern=on 5pt off 5pt] (8.2,-2)-- (9.8,-2);
\draw [line width=1pt,dash pattern=on 5pt off 5pt] (8.2,-4)-- (9.8,-4);
\draw (6.4,-1.3) node[anchor=north west] {\tiny{$\Theta_0$}};
\draw (7.6,-1.3) node[anchor=north west] {\tiny{$\Theta_1$}};
%\draw (9.7,-1.2) node[anchor=north west] {\small{$\Theta_i$}};
%\draw (10,-3.8) node[anchor=north west] {\small{$\Theta_{i+1}$}};
\draw (6.2,-3.8) node[anchor=north west] {\tiny{$\Theta_{n-1}$}};
\draw (10.2,-1) node[anchor=north west] {${\rm I}_n\ (A_{n-1})$};
\begin{scriptsize}
\fill [color=black] (0,2) circle (3.0pt);
\fill [color=black] (1,2) circle (3.0pt);
\fill [color=black] (2,2) circle (3.0pt);
\fill [color=black] (3,2) circle (3.0pt);
\fill [color=black] (4,2) circle (3.0pt);
\fill [color=black] (5,2) circle (3.0pt);
\fill [color=black] (6,2) circle (3.0pt);
\fill [color=black] (7,2) circle (3.0pt);
\fill [color=black] (2,1) circle (3.0pt);
\fill [color=black] (0,0) circle (3.0pt);
\fill [color=black] (1,0) circle (3.0pt);
\fill [color=black] (2,0) circle (3.0pt);
\fill [color=black] (3,0) circle (3.0pt);
\fill [color=black] (4,0) circle (3.0pt);
\fill [color=black] (5,0) circle (3.0pt);
\fill [color=black] (6,0) circle (3.0pt);
\fill [color=black] (3,-1) circle (3.0pt);
\fill [color=black] (0,-2) circle (3.0pt);
\fill [color=black] (1,-2) circle (3.0pt);
\fill [color=black] (2,-2) circle (3.0pt);
\fill [color=black] (3,-2) circle (3.0pt);
\fill [color=black] (4,-2) circle (3.0pt);
\fill [color=black] (2,-3) circle (3.0pt);
\fill [color=black] (2,-4) circle (3.0pt);
\fill [color=black] (5,-5) circle (3.0pt);
\fill [color=black] (4,-6) circle (3.0pt);
\fill [color=black] (5,-7) circle (3.0pt);
\fill [color=black] (3,-6) circle (3.0pt);
\fill [color=black] (1,-6) circle (3.0pt);
\fill [color=black] (0,-6) circle (3.0pt);
\fill [color=black] (-1,-5) circle (3.0pt);
\fill [color=black] (-1,-7) circle (3.0pt);
\fill [color=black] (7,-2) circle (3.0pt);
\fill [color=black] (7,-4) circle (3.0pt);
\fill [color=black] (8,-2) circle (3.0pt);
\fill [color=black] (8,-4) circle (3.0pt);
\fill [color=black] (10,-2) circle (3.0pt);
\fill [color=black] (10,-4) circle (3.0pt);
\end{scriptsize}
\end{tikzpicture}
\caption{Some reducible fibers of an elliptic fibration}\label{sing}
\end{figure}
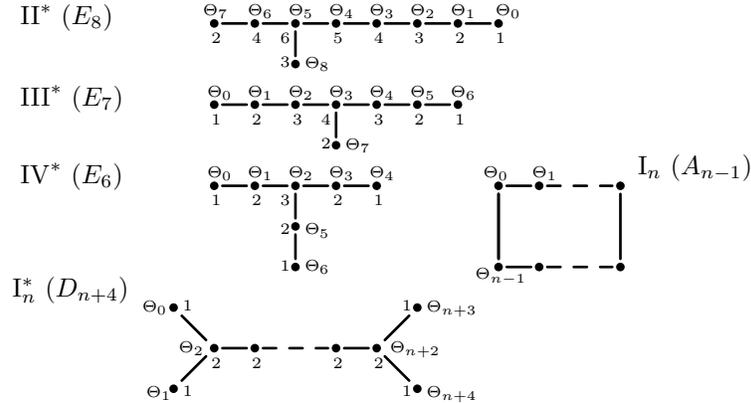

A smooth projective rational surface with
a jacobian elliptic fibration $\pi:X\to \pp^1$ 
is known to be isomorphic to the blow-up 
of the projective plane at nine points, 
eventually infinitely near. 
%(\cite[Theorem 5.6.1]{CD}). 
In fact, $X$ carries a unique elliptic fibration, induced 
by the pencil of cubics in $\pp^2$ through the nine points. 
%induces a unique elliptic fibration $\pi$ on the  
%blow up $b:X\to \pp^2$ of $\mathbb{P}^2$ at the nine base points of the pencil. 
%The elliptic fibration $\pi$ is indeed induced by the anti-canonical divisor 
%$-K_X$, whose class is $3e_0-\sum_{i=1}^9e_i$. 
%In particular, $-K_X$ is the class of the fiber of the elliptic fibration $\pi$.  
%A rational curve on $X$ is a section of the fibration if and only if it is a $(-1)$-curve. 
%Fixed a blow-down morphism $X\to \pp^2$, 
%a natural basis for the Picard group of $X$ 
%is given by $e_0,e_1,\dots,e_9$,
%where $e_0$ is the class of the pull-back  
%of a line in $\pp^2$ and the $e_i$'s are the classes of the exceptional 
%divisors satisfying $e_i\cdot e_j=-\delta_{ij}$
%(observe that $e_1,\dots,e_9$ are not necessarily irreducible 
%if some of the points are infinitely near).
Since the class of a fiber of $\pi$ is the anticanonical class,
the sections of $\pi$ are exactly the $(-1)$-curves.
%The irreducible $e_i$'s, $1\leq i\leq 9$, correspond to $(-1)$-curves, 
%thus to sections of the fibration. 
%The others $e_i$'s, $1\leq i\leq 9$, are sums of $(-2)$-curves, 
%which are irreducible components of reducible fibers, with a $(-1)$-curve.

%\begin{example} Let us consider the pencil of cubics 
%\[
%C_t:\ (x_0-x_2)(x_0-2x_1+x_2)(x_0+2x_1+x_2)+tx_0x_1x_2=0.
%\] 
%The cubic $C_0$ (resp. $C_{\infty}$), corresponding to $t=0$ (resp. $t=\infty$), 
%consists of three lines and has three double points. 
%We observe that  a singular point of the cubic $C_0$ lies on 
%the cubic $C_{\infty}$ and viceversa, 
%so we have two pairs of two infinitely near points. 
%Thus, the elliptic fibration induced by $\mathcal{P}$ 
%has two fibers of type ${\rm I}_4$ over the points $t=0$ and $t=\infty$. 
%The four components of each ${\rm I}_4$ fiber are two $(-2)$-curves 
%coming from the blow up of the infinitely near points and the strict transform of two lines. 
%There are two other reducible fibers, of type ${\rm I}_2$, which correspond to $t=\pm 1/8$, 
%indeed the cubics $C_{\pm 1/8}$ are the union of a conic and a line. 
%Thus the elliptic rational surface associated to the pencil $\mathcal{P}$ is $X_{4422}$. 
%This fibration has 8 sections: 7 of them are the exceptional divisors of the blow up of the 
%base points of the pencil and one of them is the strict transform of the line through 
%the singular point of $C_0$ which lies on $C_\infty$ and the singular point 
%of $C_\infty$ which lies on $C_0$. \end{example}

Rational elliptic surfaces with finite 
Mordell-Weil group 
have been classified by Miranda 
and Persson in~\cite{MiPe}.
The result of the classification 
is given in Table \ref{extremal}. 
In Table \ref{equa} we give 
the equation of a pencil of cubic 
curves giving rise to any such surface. 
%Some of these models are taken from~\cite{CD},
%except for ..
\begin{table}[h!] 
\begin{center}
\begin{tabular}{l|l}
Surface & Pencil of cubics \\
\midrule
$X_{22}$ & $(x_1^3+x_0^2x_2)+t x_2^3=0$\\
$X_{211}$ & $(x_1^3+x_0^2x_2+x_1^2x_2)+tx_2^3=0$\\
$X_{411}$	 & $(x_0^2x_1+x_2^3+x_1^2x_2)+tx_1x_2^2=0$\\
$X_{9111}$	 & $(x_0^2x_1+x_1^2x_2+x_2^2x_0)+tx_0x_1x_2=0$\\
$X_{33}$	 &	$x_0(x_0x_2-x_1^2)+tx_2^3=0$\\
$X_{321}$	 &	$x_0(x_0x_2-x_1^2+x_1x_2)+tx_2^3=0$\\
$X_{8211}$	 &	$(x_0-x_1)(x_0x_1-x_2^2)+tx_0x_1x_2=0$\\
$X_{44}$	 &	$x_1x_2(x_1-x_2)+tx_0^3=0$\\
$X_{431}$	 &	$x_1x_2(x_0+x_1+x_2)+tx_0^3=0$\\
$X_{222}$	 &	$x_0x_1(x_0-x_1)+t(x_1^3+2x_0x_1x_2-2x_1^2x_2-x_0x_2^2+x_1x_2^2)=0$\\
$X_{141}$	 &	$x_2(x_0x_1-x_1^2+x_0x_2)+tx_0x_1(x_0-x_1)=0$\\
$X_{6321}$ &	$(x_0+x_1)(x_1+x_2)(x_0+x_2)+tx_0x_1x_2=0$\\
$X_{11}(a)$ &	$x_1x_2(x_1-x_2)+t(x_1-ax_2)x_0^2=0,\	a \in \cc-\{0,1\}$\\
$X_{5511}$  &	$(x_1+x_2)(x_0+x_1)(x_0+x_1+x_2)+tx_0x_1x_2=0$\\
$X_{4422}$  &	$(x_0-x_2)(x_0-2x_1+x_2)(x_0+2x_1+x_2)+tx_0x_1x_2=0$\\
%$(x_1+x_2)(x_0^2+x_1^2+x_2^2-2x_0x_1+2x_0x_2+2x_1x_2)+tx_0x_1x_2=0$\\
$X_{3333}$  &	$(x_0^3+x_1^3+x_2^3)+tx_0x_1x_2=0$.
\end{tabular}
\end{center}
\vspace{0.3cm}
\caption{Pencils of cubics inducing extremal elliptic surfaces}\label{equa}
\end{table}

In what follows, it will be useful to know the intersection 
matrix for the curves with negative intersection on $X$.
We recall that, in case $MW(\pi)$ 
is finite, two sections of $\pi$ 
do not intersect by \cite[Proposition VII.3.2.]{M}.
 The intersections between the torsion sections 
 and the reducible fibers of extremal rational 
 elliptic surfaces, which can be computed by means of the 
 height formula \cite{Shio}, are summarized
 in Table \ref{int}. 
 In the following we will denote by $\Theta_i^j$ the $i$-th component,
with the notation in Figure \ref{sing}, of the $j$-th reducible fiber 
 (once a numbering of the reducible fibers is chosen). 
In particular $\Theta_0^j$ will denote the component intersecting the zero section.
 If the Mordell--Weil group is isomorphic to 
 $\zz/n\zz\oplus \zz/m\zz$, $n\geq m$, 
 we denote by $P_1$ a generator of $\zz/n\zz$ 
 and by $Q_1$ a generator of $\zz/m\zz$. 
 Moreover,  $P_j$ denotes the section which is $j$ times $P_1$ 
 with respect to the group law of the Mordell--Weil group.
 \begin{table}[h]
\[ 
\begin{array}{c|c|c}
%\hline
\mbox{ Surface } & MW(\pi) &\mbox{ Intersection properties } \\
%\hline
%X_{22}, X_{211}&  \{0\}&\\
\hline
X_{411}& \zz/2\zz&P_1\cdot\Theta_8=1\\
\hline
X_{9111}& \zz/3\zz& P_1\cdot\Theta_3=1 \\
\hline
X_{33}, X_{321}& \zz/2\zz&P_1\cdot\Theta_6^1=P_1\cdot\Theta_1^2=1 \\
\hline
X_{8211}& \zz/4\zz& P_1\cdot\Theta_2^1=P_1\cdot \Theta_1^2=1\\
\hline
X_{44}, X_{431}& \zz/3\zz&P_1\cdot\Theta_6^1=P_1\cdot \Theta_1^2=1\\
\hline
X_{222}& (\zz/2\zz)^{\oplus 2}&P_1\cdot\Theta_1^1=P_1\cdot\Theta_1^2=P_1\cdot\Theta_1^3=1\\
&&Q_1\cdot\Theta_6^1=Q_1\cdot\Theta_1^1=Q_1\cdot\Theta_0^3=1\\
\hline
X_{141}& \zz/4\zz&P_1\cdot\Theta_5^1=P_1\cdot\Theta_1^2=1 \\
\hline
X_{6321} &\zz/6\zz& P_1\cdot \Theta_1^i=1,\ i=1,2,3\\
\hline
X_{11}(a) &(\zz/2\zz)^{\oplus 2}&P_1\cdot\Theta_1^1=P_1\cdot\Theta_1^2=1\\
&&Q_1\cdot\Theta_3^1=Q_1\cdot\Theta_3^2=1\\
\hline
X_{5511}& \zz/5\zz& P_1\cdot\Theta_1^1=P_1\cdot\Theta_2^2=1\\
\hline
X_{4422}& \zz/4\zz\oplus \zz/2\zz& P_1\cdot \Theta_1^i=P_1\cdot\Theta_0^4=1,\ i=1,2,3\\ 
&& Q_1\cdot\Theta_2^1=Q_1\cdot\Theta_0^2=Q_1\cdot\Theta_1^j=1,\ j=3,4\\
\hline
X_{3333}& (\zz/3\zz)^{\oplus 2}&P_1\cdot \Theta_1^i=P_1 \cdot\Theta_0^4=1,\ i=1,2,3\\
&&Q_1\cdot \Theta_1^i=Q_1\cdot \Theta_2^2=Q_1\cdot \Theta_0^3=1,\ i=1,4\\

%\hline
\end{array}
\]
\vspace{0.1cm}
\caption{Intersections between sections and reducible fibers}\label{int}
\end{table}

\subsection{Cox rings}
The Cox ring of a complete 
normal variety with finitely generated 
and free class group $\Cl(X)$ can be defined 
as follows:
\[
\R(X):=\bigoplus_{D\in K} H^0(X,\mathcal O_X(D)),
\]
where $K$ is a subgroup of $\Div(X)$ 
such that the class map $K\to \Cl(X)$ 
is an isomorphism (see \cite[]{ADHL}).
Given a homogeneous element $f$ of $\R(X)$,
we will denote its {\em degree} with $\deg(f)\in \Cl(X)$
and we will say that $f$ is the {\em defining section} 
of a divisor $E$ if  $E=\div(f)+D$, where $D\in K$ is the 
representative for the class $\deg(f)$.
Moreover, we will say that 
$\R(X)$ is {\em generated in degree} $[D]\in \Cl(X)$
if any minimal set of generators 
of $\R(X)$ contains an element of such degree.

If $\R(X)$ is a finitely generated algebra,
then $X$ is called a {\em Mori dream space}.
By \cite[Theorem 4.2, Corollary 5.4]{ArLa} a 
rational elliptic surface $\pi:X\to \pp^1$ 
is a Mori dream space if and only if the Mordell-Weil group 
of $\pi$ is finite, or equivalently if 
the effective cone of $X$ is rational polyhedral.
%Observe that the effective cone 
%of a Mori dream space is the polyhedral cone 
%generated by the degrees of a set of generators 
%of $\R(X)$.
We observe that if $X$ is an extremal elliptic surface 
then the effective cone is generated by the 
classes of the negative curves of $X$,  
i.e. $(-1)$- and $(-2)$-curves, by \cite[Proposition 1.1]{ArLa}.

We now recall a standard result
for proving that $\R(X)$ is not generated 
in degree $[D]$, where $D$ is an effective divisor on $X$.
Let $E_1,E_2$ be two disjoint curves of $X$ 
not linearly equivalent to $D$
and $x_1,x_2\in \R(X)$ be their defining sections.
Then there is 
%an exact sequence of sheaves:
%\[
 % \xymatrix@1@C16pt{
 % 0\ar[r] & \mathcal O_X(D-E_1-E_2)\ar[r]^{f\qquad} & 
 % \mathcal O_X(D-E_1)\oplus \mathcal O_X(D-E_2)\ar[r]^-{g }& 
 % \mathcal O_X(D)\ar[r] & 0,
 % }
%\]
%where $f(s)=(sx_2,-sx_1)$ and $g(s_1,s_2)=s_1x_1+s_2x_2$.
%This induces the 
an exact cohomology sequence:
\[
  \xymatrix@1@C16pt{
  H^0(X, D-E_1)\oplus H^0(X,D-E_2)\ar[r]^-{g^*} & 
 H^0(X,D)\ar[r] & H^1(X, D-E_1-E_2),
%  \hspace{4cm}(s,t)\ar@{|->}[r] & sx_i+tx_j \\
 }
\]
where $g^*(s,t)=sx_1+tx_2$.
\begin{lemma}\label{h1}
If $g^*$ is surjective, then $\R(X)$ 
is not generated in degree $[D]$.
In particular this holds if $h^1(X,D-E_1-E_2)=0$.
\end{lemma}
%\begin{proof}
%If $g^*$ is surjective, then any element of $H^0(X,D)$ 
%can be written as a polynomial in $x_1,x_2$ and sections 
%in $H^0(X,D-E_i)$, $i=1,2$. Thus we can choose a minimal 
%generating set of $\R(X)$ which does not contain any element
%of degree $[D]$.
%\end{proof}
We finally state a proposition which is an immediate consequence of 
a result by Harbourne ~\cite[Theorem III.1]{Har}.
Observe that a nef divisor on a rational elliptic 
surface is linearly equivalent to an 
effective divisor by Riemann-Roch theorem.
Moreover, as a consequence of the following result,
the linear system associated to a nef divisor 
is base point free if and only if it has no components 
in its base locus.
\begin{proposition}\label{har}
Let $\pi:X\to \pp^1$ be a rational elliptic
surface and let $D$ be a nef
divisor of $X$.
\begin{enumerate}
\item If $-K_X\cdot D>0$, then $h^1(X,D)=0$.
\item If $-K_X\cdot D\geq 2$, then 
$|D|$ is base point free.
\item If $-K_X\cdot D=1$, then $D\sim-aK_X+P$ 
for some positive integer $a$,
where $P$ is a section of $\pi$;  
in this case $|D|$ contains $P$ in its base locus.
\item  If $-K_X\cdot D=0$, 
then $D\sim -aK_X$ for some non-negative
integer $a$ and $h^1(X,D)=a$.
\end{enumerate}
\end{proposition}
%\begin{lemma}
%Let $A,B$ be two divisors on a smooth rational surface $X$ such that $h^1(A)=h^1(B)=0$. Then $h^1(A+B)=0$.
%\end{lemma}
%\begin{proof}
%Consider the exact sequence
%\[
%0\to \Osh_X\to \Osh_X(B)\to \Osh_B(B)\to 0.
%\]
%Since $X$ is rational, then $h^1(\Osh_X)=h^2(\Osh_X)=0$, thus $h^1(\Osh_B(B))=h^1(\Osh_X(B))=0$.
%Tensoring the same sequence with $\Osh_X(A)$ and looking at the induced sequence in cohomology we obtain
%\[
%H^1(\Osh_X(A))\to H^1(\Osh_X(A+B))\to H^1(\Osh_B(A+B))\to H^2(\Osh_X(A))
%\]
%Since $h^1(A)=h^2(A)=0$, $\deg(A_{|B})\geq 0$ and $h^1(B_{|B})=0$, then 
%$h^1(A+B)=0$.
%\end{proof}

\section{Generators for the Cox rings of extremal\\ rational elliptic surfaces}
In this section we will determine a minimal 
set of generators for any extremal 
rational elliptic surface $X$.
We start proving that
the degrees of the generators of $\R(X)$ 
are at the boundary of the nef cone. 
\begin{proposition}\label{ample}
Let $\pi:X\to\pp^1$ be an extremal rational elliptic 
surface and $D$ be an ample divisor on $X$.
Then $\R(X)$ is not generated 
in degree $[D]$.
\end{proposition}
\begin{proof}
If $X\not=X_{3333}$, then it 
contains two $(-2)$-curves 
$E_1,E_2$ which either 
belong to distinct fibers of $\pi$ 
and do not intersect the same section, 
or belong to the same fiber and have distance at least two 
(in the intersection graph).
The existence of such pair of curves can be easily 
verified looking at Table \ref{extremal} and  Table \ref{int}.
The divisor $D-E_1-E_2$ is nef and 
$-K_X\cdot (D-E_1-E_2)=-K_X\cdot D>0$. 
Thus $h^1(D-E_1-E_2)=0$ by Proposition \ref{har}.

If $X=X_{3333}$ let $E_1,E_2$ be two disjoint
$(-2)$-curves and let  $B:=-K_X+D-E_1-E_2$.
Since the only curve of negative self-intersection  
intersecting both $E_1$ and $E_2$ is a section 
and $D$ is ample, then $B$ is nef.
Moreover, since $D-E_1$ is nef and 
$-K_X-E_2\sim E_2'+E_2''$, where 
$E_2'$ and $E_2''$ are the remaining 
components of the reducible fiber 
containing $E_2$, then
$B^2=(D-E_1)^2+2(D-E_1)\cdot (E_2'+E_2'')
+(E_2'+E_2'')^2\geq 4-2>0$, so that
$B$ is big.
%Since $E_i$, $i=1,2$ are components of two distinct reducible fibers, one can always assume $DE_i\leq D(-K_X)/3$, $i=1,2$, which implies $B$ is also big. 
Thus $h^1(D-E_1-E_2)=0$ by the
Kawamata-Viehweg  vanishing theorem.
In both cases the statement follows 
from Lemma \ref{h1}.
\end{proof}

%\subsection{Negative curves}
\begin{definition}
A connected and reduced curve $E$ of $X$ is a 
{\em generalized $(-1)$-curve} if there exists 
a morphism to a smooth surface contracting $E$ to a point,
a {\em generalized $(-2)$-curve} if the intersection 
graph on its irreducible components is a Dynkin diagram.
\end{definition}

It follows easily from the definition and Castelnuovo's criterion \cite[Theorem 5.7]{Ha} 
that a generalized $(-1)$-curve 
is either a $(-1)$-curve or a chain of $(-2)$-curves 
with a $(-1)$-curve at one extreme.

\begin{lemma}\label{N} 
%Let $D$ be an effective divisor with $D^2>0$ 
%and let $D^{\perp}$ be the set of negative curves 
%orthogonal to $D$.
%If $D^{\perp}$ is not empty, then 
%any connected component of $D^{\perp}$  
Let $N$ be a connected and reduced divisor of $X$
such that the intersection 
form on the irreducible components of its support 
is negative definite. 
Then $N$ is either a generalized $(-2)$-curve 
or a generalized $(-1)$-curve.
\end{lemma}
\begin{proof}
If $N$ only contains $(-2)$-curves,
then  it is a generalized 
$(-2)$-curve since its support 
is properly contained 
in a fiber of $\pi$.

Now assume that $N$ contains a $(-1)$-curve.
If $N$ contains a generalized 
$(-1)$-curve $E$ and two   
$(-2)$-curves $F,G$ %(which are necessarily disjoint) 
intersecting $E$ simply, 
then $(2E+F+G)^2=0$, giving a contradiction.
Moreover, if $E,E'$ are two $(-1)$-curves 
in $N$ and $F$ a chain of $(-2)$-curves intersecting 
both simply, then $(E+E'+F)^2=0$, again
giving a contradiction. We recall that $E$ and $E'$ 
are disjoint since they are torsion sections.
This implies that $N$ is a generalized $(-1)$-curve.
\end{proof}

\subsection{Conic bundles}  
Let $D$ be a {\em conic bundle}, that is an effective 
divisor such that $D^2=0$, $-K_X\cdot D=2$ and 
$|D|$ is base point free, so that  the morphism 
associated to $|D|$ is a rational fibration.
We now identify the conic bundles 
on $X$ giving generators for the Cox ring.

\begin{proposition}\label{proposition: D2=0}
Let $D$ be an effective divisor
on $X$ such that $D^2=0$ and $|D|$ is base point free.
Then either $D\sim aF$, $a>0$, where $F$ is a fiber of $\pi$,
or $D$ is a conic bundle.
Moreover, $\R(X)$ is generated
in degree $[D]$ exactly in the following cases:
\begin{table}[h!]
\begin{tabular}{l|l}
 $X$ & $D$\\
 \midrule
 $X_{22}, X_{211}, X_{411},X_{9111}$ & $F$\\
 \midrule
 $X_{22},X_{211}$ & $2P_0+2\sum_{i=0}^5\Theta_i+\Theta_6+\Theta_8$\\
 \midrule
 $X_{411}$ & $P_0+P_1+\Theta_0+\sum_{i=2}^6\Theta_i+\Theta_8$, \\
 & $2P_0+2\Theta_0+2\sum_{i=2}^6\Theta_i+\Theta_7+\Theta_8$, \\
 & $2P_1+2\sum_{i=2}^6\Theta_i+2\Theta_8+\Theta_0+\Theta_1$\\
 \midrule
 $X_{8211}$ & $P_0+P_1+\sum_{i\not=1}\Theta_i^1$,\\
& $P_0+P_3+\sum_{i\not=7}\Theta_i^1$,\\
 & $P_1+P_2+\sum_{i\not=3}\Theta_i^1$,\\
 & $P_2+P_3+\sum_{i\not=5}\Theta_i^1$\\
 \midrule
 $X_{9111}$ & $P_0+P_2+\sum_{i=0}^6\Theta_i$,\\
& $P_0+P_1+\sum_{i=3}^8\Theta_i+ \Theta_0$,\\
& $P_1+P_2+\sum_{i=0}^3\Theta_i+\sum_{i=6}^8\Theta_i$\\
\midrule
 $X_{33},X_{321}$ & $P_0+P_1+\sum_{i=0}^6\Theta_i^1$
\end{tabular}
\end{table}

where the notation is as in Table \ref{int}.
% where $P_0$ is the zero section, 
% $P_1$ is a generator of $MW(\pi)$ and 
% $P_j$ is the sum of $S$ $j$-times with respect to the group law. 
\end{proposition}
\begin{proof}
Since $D^2=0$ and $-K_X$ is nef, 
then by adjunction formula either 
$-K_X\cdot D=0$ and $g(D)=1$ or 
$-K_X\cdot D=2$ and $g(D)=0$. 
In the first case, by Hodge index theorem, 
$D\sim aF$, where $F$ is a fiber of $\pi$ and 
$a$ is a positive integer.
 
 Observe that $\R(X)$ has a generator in degree 
 $[F]=-K_X$ if and only if $H^0(X,-K_X)$ is not 
 generated by sections of negative curves, i.e. 
 if $\pi$ has a unique reducible fiber.
 Similarly, if $D$ is a conic bundle, $\R(X)$ 
 has a generator in degree $[D]$ if and only 
 if the rational fibration associated to $|D|$ 
 has a unique reducible fiber 
 (observe that any conic bundle has 
 at least a reducible fiber by Kleiman criterion). 
%  \begin{center}
 %\begin{figure}[h]
 % \tiny{ 
% \begin{tikzpicture}
 %\draw (0.32,0.1)-- (2.32,0.1);
 %\draw[dashed] (2.32,0.1)--(4.32,0.1);
 %\draw (4.32,0.1)-- (5.32,0.1);
 % \fill [color=red] (0.32,0.1) circle (2.5pt);
% \draw (0.3,0.5) node {$1$};
 %\fill [color=blue] (1.32,0.1) circle (2.5pt);
 %\draw (1.3,0.5) node {$1$};
 %\fill [color=blue] (2.32,0.1) circle (2.5pt);
 %\draw (2.3,0.5) node {$1$};
 %\fill [color=blue] (4.32,0.1) circle (2.5pt);
% \draw (4.3,0.5) node {$1$};
% \fill [color=red] (5.32,0.1) circle (2.5pt);
 %\draw (5.3,0.5) node {$1$};
 %\end{tikzpicture}
%}
 %\tiny{ 
% \begin{tikzpicture}
% \draw (0.32,0.1)-- (2.32,0.1);
% \draw[dashed] (2.32,0.1)--(4.32,0.1);
 %\draw (4.32,0.1)-- (6.32,0.1); 
 %\draw (5.32,0.1)-- (5.32,-0.88);
% \fill [color=red] (0.32,0.1) circle (2.5pt);
% \draw (0.35,0.5) node {$2$};
 %\fill [color=blue] (1.32,0.1) circle (2.5pt);
  % \draw (1.35,0.5) node {$2$};
 %\fill [color=blue] (2.32,0.1) circle (2.5pt);
% \draw (2.35,0.5) node {$2$};
 %\fill [color=blue] (4.32,0.1) circle (2.5pt);
 %\draw (4.35,0.5) node {$2$};
% \fill [color=blue] (5.32,0.1) circle (2.5pt);
 %\draw (5.35,0.5) node {$2$};
 %\fill [color=blue] (6.32,0.1) circle (2.5pt);
 %\draw (6.35,0.5) node {$1$};
 %\fill [color=blue] (5.32,-0.88) circle (2.5pt);
%\draw (5.3,-1.26) node {$1$};
 %\end{tikzpicture}
 %}
% \caption{Reducible fibers of conic bundles}\label{cbred}
%\end{figure}
%\end{center}

\begin{figure}[h]
\newcommand{\ndw}
 {\node[circle,draw, fill=white,inner sep=0pt,outer sep=0pt,minimum size=3pt]}
\begin{tikzpicture}[line cap=round,line join=round,>=triangle 45,x=1.0cm,y=1.0cm, scale=0.62]
\foreach \x in {-1,0,3,4}
 { \draw [line width=1pt] (\x+0.2,-6)-- (\x+0.8,-6); }
\draw [line width=1pt,dash pattern=on 5pt off 5pt] (1.2,-6)-- (2.8,-6);
\draw (-1.5,-6) node[anchor=north west] {{\tiny $-1$}};
\draw (-1.3,-5.3) node[anchor=north west] {{\tiny $1$}};
\draw (-0.3,-5.3) node[anchor=north west] {{\tiny $1$}};
\draw (0.7,-5.3) node[anchor=north west] {{\tiny $1$}};
\draw (2.7,-5.3) node[anchor=north west] {{\tiny $1$}};
\draw (3.7,-5.3) node[anchor=north west] {{\tiny $1$}};
\draw (4.7,-5.3) node[anchor=north west] {{\tiny $1$}};
\draw (4.5,-6) node[anchor=north west] {{\tiny $-1$}};
\draw (5,-6.7) node[anchor=north west] {};
\draw (4.4,-6.7) node[anchor=north west] {};
\fill [color=black] (4,-6) circle (3.0pt);
\ndw at (5,-6) {};
%\fill [color=gray] (5,-6) circle (3.0pt);
\fill [color=black] (3,-6) circle (3.0pt);
\fill [color=black] (1,-6) circle (3.0pt);
\fill [color=black] (0,-6) circle (3.0pt);
\ndw at (-1,-6) {};
%\fill [color=gray] (-1,-6) circle (3.0pt);
\end{tikzpicture}
\hspace{0.7cm}
\begin{tikzpicture}[line cap=round,line join=round,>=triangle 45,x=1.0cm,y=1.0cm, scale=0.62]
\foreach \x in {0,3,4}
 { \draw [line width=1pt] (\x+0.2,-6)-- (\x+0.8,-6); }
\draw [line width=1pt] (-.8,-5.2)-- (-.2,-5.8);
%\draw [line width=1pt] (0,-6)-- (1,-6);
%\draw [line width=1pt] (3,-6)-- (5,-6);
\draw [line width=1pt] (-.8,-6.8)-- (-.2,-6.2);
\draw [line width=1pt,dash pattern=on 5pt off 5pt] (1.2,-6)-- (2.8,-6);
\draw (4.7,-5.3) node[anchor=north west] {{\tiny $2$}};
\draw (4.6,-6) node[anchor=north west] {{\tiny $-1$}};
\draw (3.7,-5.3) node[anchor=north west] {{\tiny $2$}};
\draw (2.7,-5.3) node[anchor=north west] {{\tiny $2$}};
\draw (0.7,-5.3) node[anchor=north west] {{\tiny $2$}};
\draw (-0.2,-5.3) node[anchor=north west] {{\tiny $2$}};
\draw (-1.3,-4.3) node[anchor=north west] {{\tiny $1$}};
\draw (-1.3,-6.3) node[anchor=north west] {{\tiny $1$}};
\fill [color=black] (4,-6) circle (3.0pt);
\ndw at (5,-6) {};
%\fill [color=gray] (5,-6) circle (3.0pt);
\fill [color=black] (3,-6) circle (3.0pt);
\fill [color=black] (1,-6) circle (3.0pt);
\fill [color=black] (0,-6) circle (3.0pt);
\fill [color=black] (-1,-5) circle (3.0pt);
\fill [color=black] (-1,-7) circle (3.0pt);
\end{tikzpicture}
 \caption{Reducible fibers of conic bundles}\label{cbred}
\end{figure}
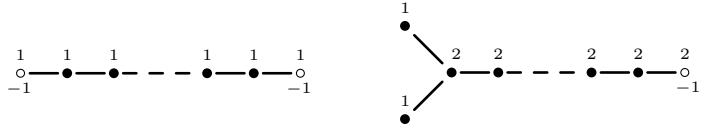

 Let $D$ be a reducible fiber of a conic bundle 
 and let $D^{red}$ be the sum of its integral components.
 Since $-K_X\cdot D=2$, then $D$ either 
 contains two $(-1)$-curves or a $(-1)$-curve 
 with multiplicity two in its support.
 Let $E$ be a curve in the support of $D^{red}$ 
 such that $D^{red}-E$ is connected. 
 Since the intersection form on 
 the components of $D^{red}-E$ is negative definite,
 then $D^{red}-E$ is either a generalized $(-1)$-curve 
 or a generalized $(-2)$-curve by Lemma \ref{N}.
 This implies that the support of $D$ 
 has the structure described by one of the diagrams 
 in Figure \ref{cbred}, where the white vertices represent 
 $(-1)$-curves and the remaining ones represent $(-2)$-curves.
 Moreover, $D$ is the unique reduced divisor with 
 the given support and with $D^2=0$:
 the multiplicity of each component is
 given by the number above each vertex 
 in Figure \ref{cbred}.
\end{proof}

\subsection{Nef and big divisors}
Let $D$ be an effective divisor without components 
in its base locus, $A,B$ be 
two smooth disjoint curves such that $D\cdot A=0$ 
and $s_A,s_B$ two sections defining $A$ and $B$. 
Then we have the following diagram:

\[
\xymatrix{
           &                                          & H^0(X,D-B)\ar[d]^{m_B}&  \\
0\ar[r] & H^0(X,D-A)\ar[r]^-{m_A} & H^0(X,D)\ar[r]^-{r_A}& H^0(A,\Osh_A) \cong \cc 
}
\]
where $m_A,m_B$ are the multiplication maps by  
$s_A$ and $s_B$ respectively, $r_A$ is the restriction to $A$  
and the horizontal sequence is exact.

\begin{lemma}\label{AB} If $h^0(X,D-B)>0$ and 
$A$ is not in the base locus of $|D-B|$, then 
$H^0(X,D)$ is generated by $H^0(X,D-A)$, $H^0(X,D-B)$ and the defining 
sections of $A$ and $B$.
\end{lemma}
\begin{proof}
Observe that  $D-A$ is effective and $r_A$ is surjective, 
since otherwise $|D|$ would contain components in its base locus.
Since $A$ is not in the base locus of $|D-B|$ 
then $r_A\circ m_B$ is also surjective.
Given $f\in H^0(X,D)$, let $r_A(f)=c\in H^0(A,\Osh_A)$ and let $g\in H^0(X,D-B)$ 
such that $r_A\circ m_B(g)=c$, 
then $f-m_B(g)\in H^0(X,D-A)$, proving the statement.
\end{proof}

%Observe that, given an ample divisor $H$ on $X$, 
%the intersections of $D-A$ and $D-B$ with $H$ 
%are strictly smaller than $D\cdot H$,  
%thus the previous Lemma suggests an inductive 
%procedure for identifying the generators of $H^0(X,D)$.
%In order to find $B$ as in the hypothesis 
%of Lemma \ref{AB} we will need the following 
%results.

Let $D$ be a nef and big divisor 
such that $|D|$ is base point free.
In what follows we will denote by $D^{\perp}$ 
the set of negative curves orthogonal to $D$.
Since $D^2>0$, the orthogonal 
complement of its class in $\Pic(X)$ 
is negative definite by Hodge index theorem.
Thus, if $D^{\perp}$ is not empty, 
the intersection form 
on the elements of $D^{\perp}$ 
is negative definite.
By Lemma \ref{N} each connected component 
of $D^{\perp}$ is either a generalized $(-1)$-curve 
or a generalized $(-2)$-curve.
%Let $N$ be a connected component
%of $D^{\perp}$.

\begin{proposition}\label{-1}
Let $D$ be a nef and big divisor   
such that $|D|$ is base point free and  
such that $D^{\perp}$ only contains $(-1)$-curves 
or only contains $(-2)$-curves,
then $\R(X)$ is not generated in degree $[D]$.
\end{proposition}
\begin{proof}
In both cases we will explain how to choose 
$A$ and $B$ in order to apply Lemma \ref{AB}.

We first assume that $D^{\perp}$ only contains 
$(-1)$-curves and let $A$ be a $(-1)$-curve 
in $D^{\perp}$. 
If $X$ contains another $(-1)$-curve $B$,
then the divisor $D-B$ is nef. 
If $-K_X\cdot (D-B)=1$, then 
$D=-aK_X+B+B'$ by Proposition \ref{har},
where $B'$ is a $(-1)$-curve. 
Since $\pi$ has always a fiber with more than two components, 
then there exists a $(-2)$-curve orthogonal 
to both $B$ and $B'$ (and thus to $D$),
contradicting the hypothesis on $D$.
Thus $-K_X\cdot (D-B)\geq 2$ 
and $D-B$ is base point free by Proposition \ref{har}.

 If $X$ only contains a $(-1)$-curve $A\in D^{\perp}$,
 then $X=X_{22}$ or $X_{211}$.
 In both cases we can choose a $(-2)$-curve 
 $B$  which does not intersect  $A$ and
intersects any other $(-2)$-curve in 
at most one simple point.  
Observe that $D-B$ is nef 
and $-K_X\cdot (D-B)=-K_X\cdot D\geq 2$,
thus $D-B$ is base point free 
by Proposition \ref{har}.

Finally, if $D^{\perp}$ only contains 
$(-2)$-curves, let $A$ be a 
 $(-2)$-curve in $D^{\perp}$ 
 and $B$ be a smooth fiber of $\pi$.
As before, $D-B$ is nef and 
base point free by Proposition \ref{har}.  
%In case $X$ contains a unique $(-1)$-curve $A$,
%i.e. either $X=X_{22}$ or $X=X_{211}$, 
%let $B$ be a $(-2)$-curve not intersecting $A$.
%The $D-B$ is nef and does not contain $A$ 
%in its base locus since otherwise,
%by Proposition \ref{har}, $D$ would 
%intersect $B$ negatively,
%contradicting the hypothesis.
\end{proof}
%By the previous lemma, we can assume that 
%$D^\perp$ contains $(-2)$-curves.
%We will say that a connected component $N$ of  
%$D^{\perp}$ containing $(-2)$-curves 
%is {\em of  type} $A_n, D_n, 
%E_6, E_7$ or $E_8$ if the intersection graph 
%of the $(-2)$-curves in $N$ 
%is of the corresponding type.
%In what follows we will identify 
%$N$ with the divisor  $\sum_{E\in N} E$.

\begin{proposition}\label{bl}
Let $D$ be a nef and big divisor
such that $|D|$ is base point free,
$B$ be an integral curve 
in  $D^{\perp}$
and $N$ be the connected
component of $D^{\bot}$
which contains $B$.
Then one of the following
holds:
\begin{enumerate}
%\item $D\sim -K_X + 2B + C$,
%where $B$ is a $(-1)$-curve
%and $C$ is a $(-2)$-curve
%with $B\cdot C=1$;
\item all the components in the base locus
of $|D-B|$ are contained in $N$;
\item $D\sim -K_X+B+C$,
where $B$ and $C$ are distinct
$(-1)$-curves.
%=|D-E|+E-B$ and $|D-E|$ 
%has no fixed components.
\end{enumerate}
%the support of the base 
%locus of $|D-B|$ is  $N$.
 \end{proposition}
\begin{proof}
%Since $X$ is a Mori dream surface,
%then $D$ is semiample by \cite[Corollary 2.38]{Hau}, so that  
%$\varphi:=\varphi_{|nD|}:X\to Y$ is a morphism onto 
%a normal surface for $n$ sufficiently big.
Since $|D|$ is base point free, then 
$\varphi:=\varphi_{|D|}:X\to Y$ is a morphism onto 
a normal surface.
Moreover, since the only negative curves of $X$ 
are $(-1)$- and $(-2)$-curves, then $Y$ has
at most ADE singularities.
Let $E$ be the exceptional divisor of $\varphi$ 
over $\varphi(B)$.
Observe that the support of $E$ is
$N$ since it is the connected component
of $D^\bot$ which contains $B$.
%Moreover either $E^2=-1$ or $E^2=-2$,
% depending on the fact that $\varphi(B)$ is a smooth
%point of $Y$ or not.
If $E^2=-1$, then $E$ is a reduced divisor 
and $N$ is a generalized $(-1)$-curve,
so that $\varphi(B)$ is a smooth point.
Otherwise, $E$ is the fundamental cycle 
of an ADE singularity, so that $E^2=-2$
and  $N$ is a generalized $(-2)$-curve
(see \cite[\S 3, Chapter III]{BP}).

We begin by showing that $E-B$ is
contained in the base locus of 
$|D-B|$.
To this aim, consider a non-trivial
decomposition $E=E_1+E_2$ 
where both $E_1$ and $E_2$ are effective
divisors.
We have that $E_i^2\leq -2$ for some $i$ 
since otherwise $E_1^2=E_2^2=-1$ 
and $E_1\cdot E_2=0$, so that $E$ would be 
disconnected.
Hence from
\[
 -2\leq E^2=E_1^2+E_2^2+2E_1\cdot E_2,
\]
we deduce $E_1\cdot E_2>0$, so that
some component $C$ of the support of 
$E_2$ must have
negative intersection with $D-E_1$,
which implies that $C$ is contained in the
base locus of $|D-E_1|$. 
If we assume that the base locus of $|D-B|$ only 
contains a proper subdivisor $F$ of $E-B$,
then we get a contradiction taking $E_1= B+F$.
This gives the claim.
%Starting with $E^0_1=B$ and   
%then taking at each step $E^i_1=E_1^{i-1}+C$,
%where $C$ is a curve in the support of $E_2^i$ 
%lying in the base locus of $D-E^1_i$,
%we obtain the claim.

We now show that $D-E$ is nef.
Let $C$ be a negative curve of $X$
with $E\cdot C >0$.
First of all we observe that if $E^2=-1$,
then $E$ is a generalized $(-1)$-curve
and $E\cdot C'\leq 0$ for any curve
$C'$ contained in its support. 
The same holds if $E^2=-2$
by~\cite[Corollary 3.6]{BP}.
Hence in both cases $C$ can not
be a component of the support $N$
of $E$.
We also have $D\cdot C>0$, since
otherwise $C$ would be contained
is some connected component
of $D^\bot$, disjoint from $N$,
giving $E\cdot C=0$.
By looking at the possible intersection
graphs for $E$ (see \cite[\S 3, Chapter III]{BP}) 
we see that 
\[
 E\cdot C \leq 2.
\]
In case $E^2=-1$ this is clear since $E$ is a reduced divisor.
If $E^2=-2$ this follows from 
the fact that there exists a fiber $F$ of $\pi$ 
such that $F-E$ is effective. 

Thus we have two consider only
two cases. If $E\cdot C=1$, then
$(D-E)\cdot C\geq 0$ due to $D\cdot C>0$.
Assume now $E\cdot C = 2$.
In this case $C^2=-2$ and 
the divisor $E+C+K_X$ is lineary equivalent
to an effective divisor by the Riemann-Roch theorem 
since $(E+C+K_X)^2-K_X\cdot(E+C+K_X)\geq 0$. 
This implies that $D\cdot (E+C)\geq 2$ 
by Proposition \ref{har}.
Hence
\[
 (D-E)\cdot C
 =
 D\cdot C-2
 =
 D\cdot (E+C)-2\geq 0,
\]
proving that $D-E$ is nef.

We now study the base locus of
$|D-E|$.
We can assume $E^2=-1$, since
otherwise $-K_X\cdot (D-E)\geq 2$
so that $|D-E|$ would be base point free 
by Proposition~\ref{har}.
Moreover, by the same proposition, the linear
system $|D-E|$ is base point free unless 
$D-E\sim -aK_X+C$, where 
$C$ is a $(-1)$-curve of $X$,
that is
\[
 D \sim -aK_X+E+C.
\]
In this case the unique component
in the base locus of $|D-E|$ is $C$. So, either we are in
case (i) of the statement or $C$ is not
contained in the support $N$ of $E$ and
in particular $E\cdot C\geq 0$.
Assume we are in this second case.
Since $E^2=-1$, then $D\cdot C=a+E\cdot C-1
=D\cdot E=0$, so that $a=1$ and $E\cdot C=0$.
Observe that $E$, which is a generalized
$(-1)$-curve, can not contain $(-2)$-curves,
since otherwise the last $(-2)$-curve
in the chain would intersect $D$ negatively
contradicting the hypothesis.
Thus we are in case (ii), concluding the proof.
% and either $C\in N$, 
%or  $C$ belongs to a distinct component of $D^{\perp}$.
% %This implies that $N\cdot S=0$ and $a=1$. 
% In the second case, if $R$ is an external $(-2)$-curve of $N$,
% then $D\cdot R=N\cdot R=-1$, contradicting the fact that 
% $D$ is nef.
 % The case when $N$ is of  type $D_n$ is similar.
%Observe that $N$ contains no $(-1)$-curves 
%by Lemma \ref{N}.
%Thus $N=\Theta_0+\Theta_1+\cdots+\Theta_{k}$, 
%$k\geq 3$,  
%where the $\Theta_i$'s have an intersection graph 
%of type $D_n$ with the notation in Table \ref{sing}.
%It can be easily checked that the divisor 
%$D':=D-\Theta_0-\Theta_1-\Theta_{k}-2\sum_{i=2}^{k-1}\Theta_i$
% is nef (again we use the fact that $-K_X\cdot D\geq 2$ to see that $D'$ has 
% non-negative intersection with the curves outside $D^{\perp}$),
% thus $|D'|$ is base point free by Proposition \ref{har}.
% In case $N$ is of type $E_6, E_7$ or $E_8$
% we can assume that the support of $N$ consists of the curves $\Theta_i$ 
% with $i=1,\dots,6$, $i=1,\dots,7$ and $i=1,\dots,8$ respectively.
% Then the following divisors are nef:
% \[
% D-2\Theta_1-3\Theta_2-2\Theta_3-\Theta_4-2\Theta_5-\Theta_6, 
% \]
% \[
% D-2\Theta_1-3\Theta_2-4\Theta_3-3\Theta_4-2\Theta_5-\Theta_6-2\Theta_7,
% \]
% \[
% D-2\Theta_1-3\Theta_2-4\Theta_3-5\Theta_4-6\Theta_5-4\Theta_6-2\Theta_7-3\Theta_8,
% \]
% and we conclude as before that they are base point free 
% by Proposition \ref{har}.
 \end{proof}

\subsection{Generators of $\R(X)$}
We now determine a minimal 
generating set for the Cox ring of any 
extremal rational elliptic surface.
\begin{definition}\label{dis} Consider the following divisors:
\begin{enumerate}
\item $(-1)$-curves or $(-2)$-curves,
\item conic bundles,
\item $-K_X$

\item nef divisors $D$ with $D^2=1$, $D\cdot (-K_X)=3$. 
%and such that $D^{\perp}$ contains a generalized $(-1)$-curve of length $9$.
\end{enumerate}
For each such divisor $D$ choose a basis of $H^0(X,D)$ and 
let $\Omega$ be the union of all such bases.
An element $f\in \Omega$ will be called a {\em distinguished section} of $\R(X)$.
A {\em distinguished polynomial} is a polynomial in the distinguished sections.
\end{definition}

\begin{lemma}\label{kab}
Let $D=-K_X+A+B$, where $A,B$ are two distinct $(-1)$-curves. 
Then $H^0(X,D)$ is generated by distinguished sections. 
\end{lemma} 
\begin{proof}
By Riemann-Roch Theorem and Proposition \ref{har} we have that  
$h^0(X,D)=3$.
Moreover, observe that $|D|$ contains $|-K_X|+A+B$ 
and there exists a conic bundle $D'$ on $X$, 
obtained by connecting $A$ and $B$ 
with a chain of $(-2)$-curves inside one fiber,
such that $D-D'$ is linearly equivalent 
to an effective divisor.
Thus $H^0(X,D)$ is generated by sections of degree $[-K_X]$, $[D']$
and by sections defining negative curves.
\end{proof}

\begin{theorem}\label{teo}
Let $\pi:X\to \pp^1$ be an extremal rational elliptic 
surface. Then the Cox ring $\mathcal R(X)$ is generated 
by distinguished sections.

%If the fibration $\pi:X\to \pp^1$ has only fibers of type ${\rm I}_n$ with 
%$n\leq 6$, then the Cox ring $\mathcal R(X)$ is generated by the defining 
%sections of $(-2)$-curves and of $(-1)$-curves.

%If $X=X_{8211}$,  then $\mathcal R(X)$ is generated by
%the defining sections of $(-2)$-curves, of $(-1)$-curves and 
%of a smooth fiber for each extremal conic bundle.

%If  $X=X_{9111}$, then $\mathcal R(X)$ is generated by
%the defining sections of $(-2)$-curves, of $(-1)$-curves, 
%of a smooth fiber for each extremal conic bundle,
%of a smooth fiber of $\pi$ 
%and by three sections defining smooth divisors $D$ on $X$ 
%such that $D^{\perp}$ is a generalized $(-1)$-curve  
%of type $A_8$.
\end{theorem}

%\begin{theorem}\label{I_n}
%If the fibration $\pi:X\to \pp^1$ has only fibers of type ${\rm I}_n$ with 
%$n\leq 6$, then the Cox ring $\mathcal R(X)$ is generated by the defining 
%sections of $(-2)$-curves and of $(-1)$-curves.
%
%If $X=X_{8211}$,  then $\mathcal R(X)$ is generated by
%the previous sections and by the defining section 
%of a smooth fiber for each extremal conic bundle.
%
%If  $X=X_{9111}$, then $\mathcal R(X)$ is generated by
%the previous sections, a section defining a smooth fiber of $\pi$ 
%and three sections defining smooth divisors $D$ on $X$ 
%such that $D^{\perp}$ is a generalized $(-1)$-curve  
%of type $A_8$.
%\end{theorem}
\begin{proof}
Let $D$ be an effective divisor of $X$.
We will prove the statement by induction 
on $d(D):=D\cdot H$, where $H$ is an ample divisor on $X$.

If  $B$ is an integral component in the base locus of $|D|$, 
then $B$ is  either a $(-1)$- or a $(-2)$-curve (since otherwise 
$h^0(X,B)>1$ by Proposition \ref{har}).
Moreover, the multiplication map 
\[
\xymatrix{
H^0(X,D-B)\ar[r]^-{\cdot f_B}& H^0(X,D)
}
\]
by the section $f_B$ defining $B$  is an isomorphism.
Since $d(D-B)<d(D)$, then $H^0(X,D)$ is generated 
by distinguished sections by induction.

If  $D^2=0$ and $|D|$ is base point free 
then $H^0(X,D)$ is generated by distinguished sections  
by Proposition \ref{proposition: D2=0}.
Thus in what follows we can assume $D$ to be nef, big and 
$-K_X\cdot D\geq 2$ by Proposition \ref{har}.
Observe that we can also assume that $D$ is not ample 
and that $D^{\perp}$ contains both $(-1)$-curves and 
$(-2)$-curves by Proposition \ref{ample}, Proposition \ref{-1} 
and the induction hypothesis.
%since otherwise $D$ would be ample 
%and $\R(X)$ is not generated in degree $[D]$ 
%by Proposition \ref{ample}.

We now assume that  $D^{\perp}$ 
has a unique connected component $N$. 
Observe that $N$ is a generalized $(-1)$-curve 
by the previous assumptions and, since it is connected,
%or a generalized $(-2)$-curve by Lemma \ref{N} 
the $(-2)$-curves in its support  
are components of a fiber $F$ of $\pi$. 
Let $A$ be a $(-2)$-curve in $N$.
%If $N$ is a $(-1)$-curve, 
%we conclude by Proposition \ref{-1}.
%Thus we can assume that $N$ 
%contains a $(-2)$-curve $A$. 

If there exists a $(-1)$-curve $B$ such that $B\cap N=\emptyset$, 
then $D-B$ is nef, so that 
$|D-B|$ contains no $(-2)$-curves 
in the base locus by Proposition \ref{har}.
The curves $A, B$ satisfy the hypothesis of Lemma \ref{AB}, 
so that $H^0(X,D)$ is generated by sections in lower 
degrees $[D-A]$ and $[D-B]$,
thus we conclude by the induction hypothesis.
On the other hand, if any $(-1)$-curve intersects $N$, 
there are two possible cases:\\
i) The external $(-2)$-curve of $N$ intersects 
a $(-1)$-curve $E$. In this case
let 
\[
Q=\sum_{i=1}^rN_i+E,
\] 
where $N_i$ 
are the components of $N$.
Observe that $Q$ is a conic bundle. 
We now show that $D-Q$ is nef, i.e. 
it has non-negative intersection with any 
negative curve $C$.
If $C$ is contained in the support of $Q$,
then $C\cdot Q=0$, so that $(D-Q)\cdot C\geq 0$.
If $C$ is not contained in the support of $Q$ 
and $C\cdot Q\leq 1$,  then clearly $(D-Q)\cdot C\geq 0$ 
since $D\cdot C> 0$.
Otherwise $C$ is not in the support of $Q$, 
$C\cdot Q=2$ and the union of $C$ with 
the $(-2)$-curves in $Q$ is the support of a fiber 
of $\pi$ of type ${\rm I}_n$, $n\geq 2$.
Since $D\cdot C=D\cdot (-K_X)\geq 2$, 
again we obtain that $(D-Q)\cdot C\geq 0$.
By Proposition \ref{har} the curves $A, Q$ 
satisfy the hypothesis of Lemma \ref{AB}, 
thus we conclude by induction hypothesis.\\
ii) The external $(-2)$-curve of $N$ does not intersects 
any $(-1)$-curve. By looking 
at Table \ref{int}, we deduce that $X$ is isomorphic to 
either $X_{9111}, X_{22}$ or $X_{211}$.
In these cases, if any $(-1)$-curve intersects $N$ and 
the external $(-2)$-curve of $N$ does not intersects 
any $(-1)$-curve, then we can assume that
 \[
 N=P_0+\sum_{i=0}^7\Theta_i.
 \]
 The morphism which contracts $N$ maps $X$ to $\pp^2$,
 thus $D=nL$, where $L$ is the pull-back of a line in $\pp^2$ via this morphism.
In this case $H^0(X,D)=\Sym^nH^0(X,L)$, so that it is generated by elements 
in $H^0(X,L)$. Since a basis of $H^0(X,L)$ is given by distinguished sections 
of type (iv), then we conclude.

We now assume that $D^{\perp}$ has at least two connected components 
$N_1,N_2$. 
%By Proposition \ref{-1} we can assume that the component $N_2$
%contains $(-2)$-curves. 
%so that it is  of type $A_n$, $n\geq 1$.
Let $A\in N_1$ and $B\in N_2$.
By Proposition \ref{bl}, either $A$ is not contained in the base locus of $|D-B|$ 
or $D\sim-K_X+A+B$.
In the first case 
$H^0(X,D)$ is generated by $H^0(X,D-A)$, $H^0(X,D-B)$ 
and by defining sections of curves with negative 
self-intersection by Lemma \ref{AB},  
thus we conclude by the induction hypothesis.
In the second case we conclude by Lemma \ref{kab}. 
\end{proof}

\begin{corollary} A minimal set of generators for $\R(X)$ 
is given by the distinguished sections  which define
 $(-1)$-curves, $(-2)$-curves,  a smooth fiber for each 
 conic bundle having a unique reducible fiber, 
 a smooth fiber of $\pi$ if it has a unique reducible fiber
and divisors of type $(iv)$ such that $D^{\perp}$ contains a 
generalized $(-1)$-curve of length nine.
\end{corollary} 
\begin{proof}
By Theorem \ref{teo} it in enough to find a set of 
distinguished sections which gives a minimal generating set of $\R(X)$.
A distinguished section of degree $[-K_X]$ is contained in such set 
if and only if the elliptic fibration has a unique reducible fiber. 
Similarly, the only conic bundles giving generators of $\R(X)$ are those 
with a unique reducible fiber, which have been classified in Proposition \ref{proposition: D2=0}. 
Finally, if $D$ is of type (iv), then $\varphi:=\varphi_{|D|}: X\to \pp^2$ is a birational morphism.
 In this case $H^0(X,D)=\varphi^*H^0(\pp^2,\Osh_{\pp^2}(1))$  is generated by sections defining negative curves unless 
 the exceptional divisor of $\varphi$  contracts a generalized $(-1)$-curve 
 of length nine to one point. A divisor $D$ with this property only exists 
 for the surfaces $X_{22}, X_{211}$ and $X_{9111}$ 
 (in the last case there are three such divisors, exchanged by the Mordell-Weil group).
 \end{proof}

%We recall that the quotient of an open 
%subset of the affine variety ${\rm Spec}\, \R(X)$ 
%for the action of a $10$-dimensional torus 
%is isomorphic to $X$ \cite{}, thus $r-\dim(I(X))-10=2.$

\section{Elliptic surfaces of
complexity one}\label{c*}

This section deals with 
extremal elliptic surfaces
$X$ which admit an action 
of the torus $T=\cc^*$.
These are the  surfaces  
$X_{22}, X_{33}, X_{44}$ and 
the surfaces $X_{11}(a)$,  where $a\in \cc-\{0,1\}$ 
(see~\cite[Proposition 9.2.17]{Du}).
The action of the torus is given as follows, where $c\in T$:
%\begin{proposition} Let $\varphi:X\to \pp^1$ be a rational elliptic surface. The following are equivalent:
%\begin{enumerate} 
%\item $\varphi$ has at most two singular fibers;
%\item the identity component of $\Aut(X)$ is isomorphic to $\cc^*$;
%\item $X$ is obtained by blowing up (nine times) $\pp^2$ at the base points of the following pencils of plane cubic curves:
\begin{center}
\begin{tabular}{ll}
%Surface &  $T$-action\\
 %\midrule
 $X_{22}$: 
 %& $z_0(x_1^3+x_0^2x_2)+z_1x_2^3=0$
 & $[x_0:x_1:x_2] \mapsto  [cx_0:x_1:c^{-2}x_2]$\\
 $X_{33}$:
 %& $z_0x_0(x_0x_2-x_1^2)+z_1x_2^3=0$
 & $[x_0:x_1:x_2] \mapsto [cx_0:x_1:c^{-1}x_2]$\\
 $X_{44}$:
 %& $z_0x_1x_2(x_1-x_2)+z_1x_0^3=0$
 & $[x_0:x_1:x_2] \mapsto [cx_0:x_1:x_2]$\\
 $X_{11}(a)$: 
 %& $z_0x_1x_2(x_1-x_2)+z_1(x_1-ax_2)x_0^2=0$
 & $[x_0:x_1:x_2] \mapsto [cx_0:x_1:x_2]$.
\end{tabular}
\end{center}
\vspace{3mm}

%
%\begin{center}
%\begin{tabular}{l|r|l}
%Type & Pencil of cubics& $T$-action\\
% \midrule
% ${\rm II}^*\, {\rm II}$ 
% & $z_0(x_1^3+x_0^2x_2)+z_1x_2^3=0$
% & $[x_0:x_1:x_2] \mapsto  [cx_0:x_1:c^{-2}x_2]$\\
% ${\rm III}^*\, {\rm III}$
% & $z_0x_0(x_0x_2-x_1^2)+z_1x_2^3=0$
% & $[x_0:x_1:x_2] \mapsto [cx_0:x_1:c^{-1}x_2]$\\
% ${\rm IV}^*\, {\rm IV}$
% & $z_0x_1x_2(x_1-x_2)+z_1x_0^3=0$
% & $[x_0:x_1:x_2] \mapsto [cx_0:x_1:x_2]$\\
% ${\rm 2I}_0^*$ 
% & $z_0x_1x_2(x_1-x_2)+z_1(x_1-ax_2)x_0^2=0$
% & $[x_0:x_1:x_2] \mapsto [cx_0:x_1:x_2]$.
%\end{tabular}
%\end{center}
%\vspace{3mm}
We will compute the Cox ring 
of such surfaces by means 
of~\cite[Theorem 3.18]{Hau}. 
We recall that a fixed point for the $T$-action
is called {\em elliptic} if it lies 
in the closure of infinitely
many orbits.
In order to compute the 
Cox ring we proceed in two steps, following 
the technique explained
in~\cite[\S 3.3]{Hau}:
first of all we produce a
$T$-equivariant blow-up
$\phi: X'\to X$ such that
$X'$ does not contain any 
elliptic fixed point,
then we construct
the intersection graph of
the negative curves of $X'$,
i.e. the {\em Orlik-Wagreich graph} of $X'$.
This information is enough
to give a presentation of
the Cox ring of $X'$, from
which one obtains
a presentation for the Cox
ring of $X$ putting equal to $1$ the 
variables corresponding to the exceptional 
divisors of $\phi$. We will discuss in detail 
one example, the other ones being similar.
\subsection{The surface $X_{22}$}
The surface $X_{22}$ contains 
one elliptic fixed point
at the cusp of the 
fiber of  type ${\rm II}$, whose
plane model is the curve
of equation $x_1^3+x_2x_0^2=0$,
and a curve of fixed points 
which corresponds to the the trivalent component 
of the fiber of type ${\rm II}^*$. 
Blowing-up the cusp three times
we obtain a surface $X_{22}'$ 
which does not contain elliptic points. 
The intersection graph of its negative 
curves is given in Figure \ref{X_22}.
%\begin{tikzpicture}
% \node[circle,draw, fill=gray!20] (F1) at (0,0) {$-2$};
 %\node[circle,draw, fill=gray!20] (2) at (1,1) {$-2$};
% \node[circle,draw, fill=gray!20] (3) at (2,1) {$-2$};
% \node[circle,draw, fill=gray!20] (4) at (3,1) {$-2$};
% \node[circle,draw, fill=gray!20] (5) at (4,1) {$-2$};
 %\node[circle,draw, fill=gray!20] (6) at (5,1) {$-2$};
% \node[circle,draw, fill=gray!20] (7) at (7,1) {$-1$};
% \node[circle,draw] (8) at (9,1) {$-6$};
% \node[circle,draw, dashed] (F2) at (10,0) {$-1$};
 %\node[circle,draw, fill=gray!20] (10) at (1.5,0) {$-2$};
 %\node[circle,draw] (11) at (7,0) {$-1$};
% \node[circle,draw, dashed] (12) at (8.5,0) {$-2$};
% \node[circle,draw, fill=gray!20] (13) at (1,-1) {$-2$};
% \node[circle,draw, fill=gray!20] (14) at (2,-1) {$-2$};
 %\node[circle,draw] (15) at (7,-1) {$-1$};
% \node[circle,draw, dashed] (16) at (9,-1) {$-3$};
 %\draw (F1) -- (2) -- (3) -- (4) -- (5) -- (6) -- (7) -- (8) -- (F2);
 %\draw (F1) -- (10) -- (11) -- (12) -- (F2);
% \draw (F1) -- (13) -- (14) -- (15) -- (16) -- (F2);
%\end{tikzpicture}
The four right-hand side  
vertices come from the blow-up of the cusp 
(the $(-1)$-curve there is pointwise fixed),
the left hand side is 
the graph of the fiber of type ${\rm II}^*$,
while the central $(-1)$-curves 
are orbits whose self-intersections 
are uniquely determined by the following conditions 
on the Hirzebruch-Jung continued fractions: 
\[
[2,2,2,2,2,1,6] =0,\ [2,1,2]=0,\ [2,2,1,3]=0.
\]
The partial quotients of the previous 
continued fractions 
give the orders of
the isotropy subgroups of
the $T$-action on the generic
points of the negative curves of $X_{22}'$ 
according to~\cite[pag. 31]{Hau}.
This allows to write down a presentation for 
$\R(X_{22}')$ by~\cite[Theorem 3.18]{Hau} 
and of $\R(X_{22})$ by putting equal to $1$ the 
variables corresponding to the exceptional 
divisors of the blow-up $X_{22}'\to X_{22}$ 
(those corresponding 
to vertices with white circle in Figure \ref{X_22}).
The ring $\R(X_{22})$ is given 
in Table \ref{cox}, where the variable 
$T_7$ is
%\[
% \cc[T_1,\dots,T_{14},S_1,S_2]/
% (T_1T_2^2T_3^3T_4^4T_5^5T_6^6T_7
% +
% T_8T_9^2T_{10}
% +
% T_{11}T_{12}^2T_{13}^3T_{14}
% )
%\]
%\[
% \R(X_{22})
 %\cong
 %\cc[T_1,\dots,T_{12},S_1]/
 %(T_1T_2^2T_3^3T_4^4T_5^5T_6^6T_7
 %+
 %T_8T_9^2
 %+
 %T_{10}T_{11}^2T_{12}^3
 %),
%\]
the defining section of the left trivalent vertex in the graph 
and 
$T_8,\ldots, T_{13}, T_4, T_1,T_3,T_6,T_5,T_2$ are the defining 
%$T_{4},T_8,\dots, T_{12}, T_{13},T_1, T_3,T_2,T_5,T_6$ are the defining 
sections of the remaining black vertices,
ordered row by row from the left to the right. 
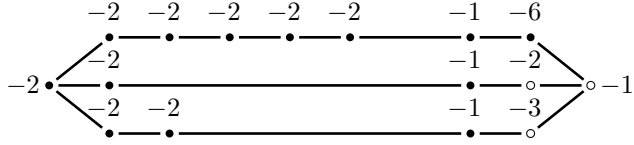
\begin{figure}[h!]
\begin{center}
\begin{tikzpicture}[scale=0.8]
\newcommand{\ndw}
 {\node[circle,draw, fill=white,inner sep=0pt,outer sep=2pt,minimum size=3pt]}
 \newcommand{\nd}
 {\node[circle,fill=black,inner sep=0pt,outer sep=2pt,minimum size=3pt]}
 \path (0,0) coordinate (origin);
 \foreach \x/\y/\i/\j in
{1/.8/1/-2,2/.8/2/-2,3/.8/3/-2,4/.8/4/-2,5/.8/5/-2,7/.8/6/-1,8/.8/7/-6}
{
  \nd (E\i) at (\x,\y) {};
  \node[above] at (\x-0.1,\y+0.1) {$\j$};
 }
 \foreach \x/\y/\i/\j in 
 {8/0/11/-2}
{
  \ndw (E\i) at (\x,\y) {};
  \node[above] at (\x-0.1,\y+0.1) {$\j$};
 }
  \foreach \x/\y/\i/\j in {1/0/9/-2,7/0/10/-1}
 {
  \nd (E\i) at (\x,\y) {};
  \node[above] at (\x-0.1,\y+0.1) {$\j$};
 }
 \foreach \x/\y/\i/\j in {1/-.8/12/-2,2/-.8/13/-2,7/-.8/14/-1}
 {
  \nd (E\i) at (\x,\y) {};
  \node[above] at (\x-0.1,\y+0.1) {$\j$};
 }
  \foreach \x/\y/\i/\j in {8/-.8/15/-3}
  {
  \ndw (E\i) at (\x,\y) {};
  \node[above] at (\x-0.1,\y+0.1) {$\j$};
 }
 \nd (S1) at (0,0) {};
 \ndw (S2) at (9,0) {};
 \node[left] at (0,0) {$-2$};
 \node[right] at (9,0) {$-1$};
 \draw[line width=1pt] (S1) -- (E1) -- (E2) -- (E3) -- (E4)-- (E5) -- (E6) --
(E7) -- (S2);
 \draw[line width=1pt] (S1) -- (E9) -- (E10) -- (E11) -- (S2);
 \draw[line width=1pt] (S1) -- (E12) -- (E13) -- (E14) -- (E15) -- (S2);
\end{tikzpicture}
\end{center}
\caption{The surface $X_{22}$} \label{X_22}
\end{figure}

We observe that the sections $x_1, x_0, x_2, x_1^3+x_0^2x_2$ 
define the images in $\pp^2$ of the curves defined by $T_2, T_3, T_1$ and 
$T_4$ respectively.
The unique relation of the Cox ring corresponds to 
the relation among the three cubics 
$x_1^3,\ x_0^2x_2,\ x_1^3+x_0^2x_2$,
which are elements of the pencil of cubics having $x_2=0$ as inflectional 
tangent in $(1:0:0)$ and which have a cusp at $(0:0:1)$ whose principal tangent is $x_0=0$.

\subsection{The surface $X_{33}$}
The surface $X_{33}$ contains
one elliptic fixed point,
the singular point
of the fiber of type ${\rm III}$, 
and a curve of fixed points,
given by the trivalent  
component of the fiber 
of type ${\rm III}^*$.
Blowing-up the elliptic point twice 
we obtain the surface $X'$ 
whose intersection 
graph of negative curves 
is given in Figure \ref{X_33}.
%The four right-hand side  
%vertices in the figure come 
%from the blow-up of the elliptic point, 
%the left  hand side is 
%the graph of the fiber of type ${\rm III}^*$,
%while the central $(-1)$-curves 
%are orbits whose self-intersections 
%In this case the self-intersections of the 
%central curves are determined by the following 
%Hirzebruch-Jung continued fractions 
%\[
%[2,2,2,1,4] =0,\ [2,1,2]=0,\ [2,2,2,1,4]=0.
%\]
%Observe that the upper and lower $(-1)$-curves in the graph 
%are the two sections of the fibration. 
The ring $\R(X_{33})$ is thus computed as before and 
is  given in Table \ref{cox} where the variable 
$T_6$ is
the defining section of the left trivalent vertex in the graph 
and 
$T_1, T_{11}, T_{12}, T_{13}, T_2, T_5,T_4,T_7,T_8,T_9,T_{10}, T_3$ are the defining 
sections of the remaining black vertices,
ordered row by row from the left to the right.

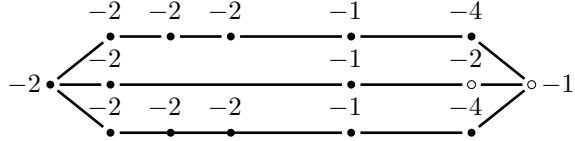
\begin{figure}[h!]
\begin{center}
\begin{tikzpicture}[scale=0.8]
\newcommand{\ndw}
 {\node[circle,draw, fill=white,inner sep=0pt,outer sep=2pt,minimum size=3pt]}
 \newcommand{\nd}
 {\node[circle,fill=black,inner sep=0pt,outer sep=2pt,minimum size=3pt]}
 \path (0,0) coordinate (origin);
 \foreach \x/\y/\i/\j in
{1/.8/1/-2,2/.8/2/-2,3/.8/3/-2,5/.8/4/-1,7/.8/5/-4}
{
  \nd (E\i) at (\x,\y) {};
  \node[above] at (\x-0.1,\y+0.1) {$\j$};
 }
 \foreach \x/\y/\i/\j in 
 {7/0/11/-2}
{
  \ndw (E\i) at (\x,\y) {};
  \node[above] at (\x-0.1,\y+0.1) {$\j$};
 }
  \foreach \x/\y/\i/\j in {1/0/9/-2,5/0/10/-1}
 {
  \nd (E\i) at (\x,\y) {};
  \node[above] at (\x-0.1,\y+0.1) {$\j$};
 }
 \foreach \x/\y/\i/\j in {1/-.8/12/-2,2/-.8/13/-2,3/-.8/13/-2,5/-.8/13/-1,7/-.8/14/-4}
 {
  \nd (E\i) at (\x,\y) {};
  \node[above] at (\x-0.1,\y+0.1) {$\j$};
 }
 \nd (S1) at (0,0) {};
 \ndw (S2) at (8,0) {};
 \node[left] at (0,0) {$-2$};
 \node[right] at (8,0) {$-1$};
 \draw[line width=1pt] (S1) -- (E1) -- (E2) -- (E3) -- (E4)-- (E5) --(E6)-- (S2);
 \draw[line width=1pt] (S1) -- (E9) -- (E10) -- (E11) -- (S2);
 \draw[line width=1pt] (S1) -- (E12) -- (E13) -- (E14)  -- (S2);
\end{tikzpicture}
\end{center}
\caption{The surface $X_{33}$} \label{X_33}
\end{figure}
\subsection{The surface $X_{44}$}
The surface $X_{44}$ contains
one elliptic fixed point
which is the singular point
of the fiber of type ${\rm IV}$.
Blowing-up this point once
we obtain the surface $X_{44}'$
whose graph of negative
curves is given in Figure \ref{X_44}. 
The ring $\R(X_{44})$ is thus computed as before and 
is  given in Table \ref{cox}
where the variable 
$T_1$ is
the defining section of the left trivalent vertex in the graph 
and 
$T_5$,$T_6$,$T_7$,$T_2$,$T_8$,$T_9$,$T_{10}$,$T_3$,$T_{11}$,$T_{12}$, $T_{13}$,$T_4$ are the defining 
sections of the remaining black vertices,
ordered row by row from the left to the right. 

\begin{figure}[h!]
\begin{center}
\begin{tikzpicture}[scale=0.8]
\newcommand{\ndw}
 {\node[circle,draw, fill=white,inner sep=0pt,outer sep=2pt,minimum size=3pt]}
 \newcommand{\nd}
 {\node[circle,fill=black,inner sep=0pt,outer sep=2pt,minimum size=3pt]}
 \path (0,0) coordinate (origin);
 \foreach \x/\y/\i/\j in
{1/.8/1/-2,2/.8/2/-2,4/.8/3/-1,6/.8/4/-3}
{
  \nd (E\i) at (\x,\y) {};
  \node[above] at (\x-0.1,\y+0.1) {$\j$};
 }
% \foreach \x/\y/\i/\j in 
 %{7/0/11/-2}
%{
  %\ndw (E\i) at (\x,\y) {};
  %\node[above] at (\x-0.1,\y+0.1) {$\j$};
 %}
  \foreach \x/\y/\i/\j in {1/0/9/-2,2/0/10/-2,4/0/11/-1,6/0/12/-3}
 {
  \nd (E\i) at (\x,\y) {};
  \node[above] at (\x-0.1,\y+0.1) {$\j$};
 }
 \foreach \x/\y/\i/\j in {1/-.8/12/-2,2/-.8/13/-2,4/-.8/14/-1,6/-.8/14/-3}
 {
  \nd (E\i) at (\x,\y) {};
  \node[above] at (\x-0.1,\y+0.1) {$\j$};
 }
  
 \nd (S1) at (0,0) {};
 \ndw (S2) at (7,0) {};
 \node[left] at (0,0) {$-2$};
 \node[right] at (7,0) {$-1$};
 \draw[line width=1pt] (S1) -- (E1) -- (E2) -- (E3) -- (E4)-- (S2);
 \draw[line width=1pt] (S1) -- (E9) -- (E10) -- (E11) -- (S2);
 \draw[line width=1pt] (S1) -- (E12) -- (E13) -- (E14)  -- (S2);
\end{tikzpicture}
\end{center}
\caption{The surface $X_{44}$} \label{X_44}
\end{figure}
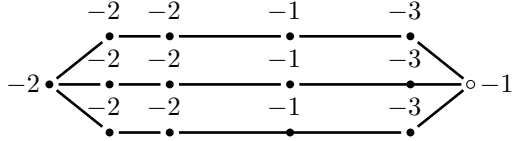

\subsection{The surface $X_{11}(a)$}
 In this case we have a one parameter
family of surfaces $X_{11}(a)$, $a\in \cc-\{0,1\}$.
%A smooth fiber of $X_{11}(a)$
%is an elliptic curve of $j$-invariant
%\[
% 256\frac{(a^2-a+1)^3}{a^2(a-1)^2}.
%\]
%In particular the fibration
%is locally isotrivial in 
%a neighborhood of any smooth
%fiber.
The surface $X_{11}(a)$ does
not contain elliptic fixed
points and its graph of negative
curves is given in Figure \ref{X_11a}  where the variables 
$T_{12}$ and $T_4$ are
the defining sections of the left fourvalent vertex and of the right fourvalent vertex respectively
and 
$T_1, T_7, T_6, T_2, T_9, T_8,T_3,T_{11},T_{10},T_{13},T_{14}, T_5$ are the defining 
sections of the remaining vertices,
ordered row by row from the left to the right. 
In this case the Cox ring, which is given 
in Table \ref{cox}, has two 
relations in the same degree. 
These relations correspond to the relations 
among the four lines $x_1=0$, $x_2=0$, 
$x_1-x_2=0$,  $x_1-ax_2=0$ of the pencil 
of lines in $\mathbb{P}^2$ through $(1:0:0)$.
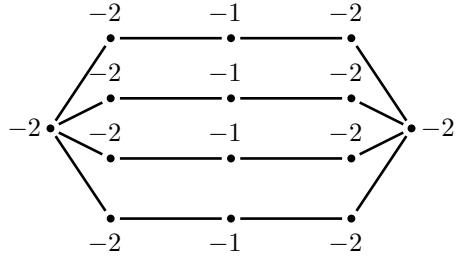
\begin{figure}[h!]
\begin{center}
\begin{tikzpicture}[scale=0.8]
\newcommand{\ndw}
 {\node[circle,draw, fill=white,inner sep=0pt,outer sep=2pt,minimum size=3pt]}
 \newcommand{\nd}
 {\node[circle,fill=black,inner sep=0pt,outer sep=2pt,minimum size=3pt]}
 \path (0,0) coordinate (origin);
 \foreach \x/\y/\i/\j in
{1/1.5/1/-2,3/1.5/2/-1,5/1.5/3/-2}
{
  \nd (E\i) at (\x,\y) {};
  \node[above] at (\x-0.1,\y+0.1) {$\j$};
 }
% \foreach \x/\y/\i/\j in 
 %{7/0/11/-2}
%{
  %\ndw (E\i) at (\x,\y) {};
  %\node[above] at (\x-0.1,\y+0.1) {$\j$};
 %}
  \foreach \x/\y/\i/\j in {1/0.5/4/-2,3/0.5/5/-1,5/0.5/6/-2}
 {
  \nd (E\i) at (\x,\y) {};
  \node[above] at (\x-0.1,\y+0.1) {$\j$};
 }
 \foreach \x/\y/\i/\j in {1/-.5/7/-2,3/-.5/8/-1,5/-.5/9/-2}
 {
  \nd (E\i) at (\x,\y) {};
  \node[above] at (\x-0.1,\y+0.1) {$\j$};
 }
  \foreach \x/\y/\i/\j in {1/-1.5/10/-2,3/-1.5/11/-1,5/-1.5/12/-2}
 {
  \nd (E\i) at (\x,\y) {};
  \node[below] at (\x-0.1,\y-0.1) {$\j$};
 }
 \nd (S1) at (0,0) {};
 \nd (S2) at (6,0) {};
 \node[left] at (0,0) {$-2$};
 \node[right] at (6,0) {$-2$};
 \draw[line width=1pt] (S1) -- (E1) -- (E2) -- (E3) -- (S2);
 \draw[line width=1pt] (S1) -- (E4) -- (E5) -- (E6) -- (S2);
  \draw[line width=1pt] (S1) -- (E7) -- (E8) -- (E9) -- (S2);
 \draw[line width=1pt] (S1) -- (E10) -- (E11) -- (E12)  -- (S2);
\end{tikzpicture}
\end{center}
\caption{The surface $X_{11}(a)$} \label{X_11a}
\end{figure}

 \section{Relations of the Cox ring}
In this section we will determine the ideal of relations 
$I(X)$ of $\R(X)$, where $X$ is an extremal 
rational elliptic surface, following
\cite[Construction 3.1]{BHK}. 

\subsection{Algebraic preliminaries}
Let $\cc[T,S]$ be a polynomial ring in $m+n$ variables 
$T_1,\dots, T_m,S_1,\dots,S_n$ graded by an abelian group 
$K_{T}\oplus K_{S}$, where $K_S$ is free.
We assume that the degrees of the $S$ variables give a basis of $K_S$.
In what follows $\cc[T]$ denotes the polynomial ring in the first $m$ variables graded 
by $K_{T}$, $\cc[T,S^{\pm 1}]$ is the localization of $\cc[T,S]$ with respect 
to all the $S$ variables and  $\cc[T,S^{\pm 1}]_0$ is its degree zero part 
with respect to the $K_S$ grading. 
Assume the following diagram of homomorphisms is given:
\[
\xymatrix{
I\ar[r]& \cc[T,S]\ar[r]^{\psi} \ar[d]_{\iota}& R\\
& \cc[T,S^{\pm 1}] &  \cc[T]\ar[u]_{\sigma}\ar[dl]^{\alpha}\ar[ul]_{\beta}\\
& \cc[T,S^{\pm 1}]_0\ar[u]^{\iota_0} & J\ar[u]
}
\]
where $\psi$ is a graded surjective homomorphism to a $K_T\oplus K_S$-graded 
domain $R$ with kernel $I$, $\sigma$ is a $K_T$-graded homomorphism 
with kernel $J$, $\iota$ and $\iota_0$ are the natural inclusions, $\sigma=\psi\circ \beta$ and 
\[
\iota_0\circ\alpha(T_i)=\iota\circ\beta(T_i)=T_i\cdot m_i(S),
\]
where $m_i(S)\in \cc[S]$ is the unique monomial in the variables $S$ such that 
$T_i\cdot m_i(S)$ has $K_S$ degree equal to zero (the uniqueness follows from the hypothesis 
on the degrees of the $S$ variables).  We observe that $\alpha$ is an isomorphism.
\begin{proposition}\label{ideal}
 Let $J'\subset \cc[T,S]$ be the extension and contraction of the ideal $\alpha(J)$ (by $\iota_0$ and $\iota$). Then $J'\subset I$.
 \end{proposition}
 \begin{proof}
 Observe that $\beta(J)\subset I$ since $\sigma=\psi\circ \beta$. Moreover, from the equality
 \[
 \beta(J)\cdot \cc[T,S^{\pm 1}]=\alpha(J)\cdot \cc[T,S^{\pm 1}]
 \]
 we get that $J'$ is contained in the saturation of $I$ with respect to the variables $S$.
 Since $R$ is a domain, then $I$ is saturated, hence we get the statement.
 \end{proof}
  
\subsection{Cox rings of blow-ups}
Let  $f:X\to Y$ be a blow-up
morphism, where $X$ and $Y$ are smooth Mori dream spaces.
Let $T_1,\dots,T_m,S_1,\dots,S_n$ be a minimal set of generators 
of $\R(X)$, where $S_1,\dots,S_n$ define the 
irreducible components of the exceptional divisor of $f$.
The polynomial ring $\cc[T,S]$ is graded by 
$\Cl(X)=K_T\oplus K_S$,
where $K_T=f^*\Cl(Y)$ and $K_S$  
is the free abelian group generated by the 
degrees of the $S$ variables.

In this setting we take $R=\R(X)$ 
and $\psi:\cc[T,S]\to \R(X)$ to be the natural quotient 
homomorphism with kernel $I(X)$. 
Observe that, since $\beta(T_i)$ has degree zero 
with respect to the $K_S$-grading,
then $\psi\circ \beta(T_i)=f^*(s_i)$ via  the
pull-back homomorphism $f^*:\R(Y)\to \R(X)$.
Thus we obtain the following result.

 \begin{corollary}\label{J} Let $J$ be the kernel of  the 
$K_T$-graded homomorphism
\[
\phi:\cc[T]\to \R(Y),\quad T_i\mapsto s_i.
\]
Then the ideal $J'$ defined before is contained in $I(X)$ and 
$J'=I(X)$ if $\dim J'=\dim I(X)$.
\end{corollary}
\begin{proof}
By definition $\sigma=f^*\circ \phi$.  
Since $f^*$ is injective, the first claim in the statement 
follows from Proposition \ref{ideal}.
By \cite[Proposition 3.3]{BHK} $J'$ is prime since $J$ is prime. 
Thus the second claim follows from the fact that $I(X)$ is prime.
\end{proof}
 We recall that the quotient of an open subset of the affine variety 
 ${\rm Spec}\,\R(X)$ for the action of a torus of dimension 
 $\rk \Cl(X)$  is isomorphic to $X$ \cite[]{ADHL}, thus 
 $\dim I(X)=\rk\Cl(X)+2$.

\subsection{The ideal of relations of rational elliptic surfaces}
We recall that any rational elliptic surface $X$ is a blow-up 
of the projective plane at nine points, eventually infinitely near.
Fixed a blow-up morphism $X\to \pp^2$, 
a natural basis for the Picard group of $X$ 
is given by $e_0,e_1,\dots,e_9$,
where $e_0$ is the pull-back 
of the hyperplane class of $\pp^2$
and the $e_i$'s with $i\geq 1$ are the classes of the 
exceptional divisors with $e_i\cdot e_j=-\delta_{ij}$ 
(observe that the exceptional divisors with classes $e_1,\dots,e_9$ 
are not necessarily irreducible if there are infinitely near points).
Thus in this case $K_T\cong \zz$ and $K_S\cong \zz^9$.
 
\begin{example}
We will apply Corollary \ref{J}
to compute the ideal of relations of the Cox ring 
of the surface $X=X_{411}$.
We recall that  $X$ has 
an elliptic fibration with a unique reducible fiber of type 
${\rm I}_4^*$ and Mordell-Weil group of order two.
 The surface can be obtained from the following 
pencil of cubic curves:
\[
(x_0^2x_1+x_2^3+x_1^2x_2)+tx_1x_2^2=0
.\]
Let $l_1: x_1=0$, $l_2: x_2=0$ and $p=(1:0:0)$ 
be their intersection point. 
We call $C$ the smooth cubic given 
by $t=0$.
Observe that the line $l_2$ is 
tangent to $C$ in $q=(0:1:0)$ and the line 
$l_1$ is the inflectional tangent of $C$ at $p$.  
The surface $X$ is obtained blowing-up 
$\mathbb{P}^2$ five times at $p$ 
(that is at $p$ and at four points infinitely near to it)
and four times at $q$. 
We denote by $e_1,\dots,e_5$ the classes of 
the exceptional divisors over $p$ and
by $e_6,\dots,e_9$ those over $q$,
where $e_i\cdot e_j=-\delta_{ij}$ 
and all the $e_i$'s are reducible except 
for $e_5$ and $e_9$, which are taken to be  
the exceptional divisors of the last blow-up 
over $p$ and $q$. 
%The classes of the strict transforms of the lines 
%$l_1$ and $l_2$ are given by $H-E_1-E_2-E_3$ 
%and $H-E_1-E_6-E_7$ respectively.
By Theorem \ref{teo} a set of generators for the Cox ring of $X$ is
given by 
\begin{enumerate}[$\bullet$]
\item the sections $T_{1}, T_{2}$ defining the 
proper transforms of the lines $l_1$ and $l_2$, whose degrees are:
\[
[\Theta_1]=e_0-e_1-e_2-e_3,\quad [\Theta_5]=e_0-e_1-e_6-e_7;
\]
\item three sections $T_{3}, T_{4}, T_{5}$ defining 
smooth fibers of the conic bundles whose degrees are:
\[
\begin{array}{l}
[\Theta_0+\Theta_8+\sum_{i=2}^6\Theta_i+P_0+P_1]=e_0-e_6,\\[3pt]
[2P_0+2\Theta_0+2\sum_{i=2}^6\Theta_i+\Theta_7+\Theta_8]=2e_0-e_6-e_7-e_8-e_9,\\[3pt]
  [2P_1+2\Theta_8+2\sum_{i=2}^6\Theta_i+\Theta_0+\Theta_1]=3e_0-(e_1+\cdots+e_5)-2e_6;
\end{array}
\]
\item the section $T_{6}$ defining the proper transform of the cubic $C$, of degree $3e_0-\sum_{i=1}^9e_i$.
\item the sections $T_7,\dots,T_{15}$ defining the irreducible 
components of the exceptional divisors 
over $p$ and $q$, whose degrees are:
\[
\begin{array}{llll}
[\Theta_0]=e_4-e_5 & [\Theta_i]=e_{5-i}-e_{6-i} & [P_0]=e_5 &(i=2,3,4)\\[3pt]
 [\Theta_6]=e_7-e_8 & [\Theta_7]=e_6-e_7 &  [\Theta_8]=e_8-e_9 &  [P_1]=e_9
\end{array}
\]

\end{enumerate}
As a smooth fiber of the first 
conic bundle we can choose 
the proper transform of the line $x_0=0$.
Moreover, the proper transform of the conic $x_0^2+x_1x_2=0$ 
 is smooth of degree $2e_0-e_6-e_7-e_8-e_9$ 
 and the proper transform of the cubic $x_0^2x_1+x_2^3=0$ 
 is smooth of degree $3e_0-(e_1+\cdots+e_5)-2e_6$,
 thus their defining sections can be taken to be 
generators of $\R(X)$.
With these choices, the sections $s_1,\dots,s_{6}$ defining the image 
of the curves defined by $T_{1},\dots, T_{6}$ are the following:
\[
\begin{array}{c}
s_1=x_1,\ s_2=x_2,\  s_3=x_0,\ s_4=x_0^2+x_1x_2,\\ 
s_5=x_0^2x_1+x_2^3,\ s_6=x_0^2x_1+x_2^3+x_1^2x_2.
\end{array}
\]
A computation with Magma \cite{Magma} shows that 
a set of generators for the ideal $J$ is:
\[
T_{1}T_{2} + T_{3}^2 - T_{4},\ 
T_{1}T_{4} + T_{2}^3 - T_{6},\ 
T_{1}T_{3}^2 + T_{2}^3 - T_{5},
\]
and gives a set  of generators for the ideal $J'$.
This set is the one given in Table \ref{cox}, in fact $J'=I(X)$ 
since a computation with the same program shows that $\dim J'= 12=\dim I(X)$.
%\[
%	T_{7}^2T_{8}^3T_{9}^4T_{10}^2T_{14}T_{15}^2T_{2}^3 - T_{10}^2T_{11}T_{6} + T_{1}T_{4},\quad 
%	-T_{7}^2T_{8}T_{6} + T_{10}T_{5} + T_{11}^2T_{12}^3T_{13}^4T_{14}^3T_{15}^2T_{1}^2T_{2},
%\]
%\[
 %	T_{7}^4T_{8}^4T_{9}^4T_{10}T_{14}T_{15}^2T_{2}^3 - T_{11}^2T_{12}T_{5} + T_{1}T_{3}^2,\quad 
%	-T_{7}^2T_{8}T_{4} + T_{10}T_{3}^2 + T_{11}^4T_{12}^4T_{13}^4T_{14}^3T_{15}^2T_{1}T_{2},
%\]	
%\[ 
%	 -T_{7}^2T_{8}^3T_{9}^4T_{10}T_{11}^2T_{12}^3T_{13}^4T_{14}^4T_{15}^4T_{1}T_{2}^4 - T_{3}^2T_{6} + T_{4}T_{5}.
 %\]
\end{example}

 \begin{remark}
 The set of generators of the ideal $J$ given above is induced by geometric relations among curves in $\mathbb{P}^2$. 
 Infact the three conics $s_1s_2$, $s_3^2$ and $s_4$ 
 belong to the pencil of conics whose tangent in $(0:1:0)$ is  $x_2=0$ 
 and whose tangent in $(0:0:1)$ is $x_1=0$, 
 thus there exists a linear relation among them, which indeed corresponds to the first generator of $J$. 
 The three cubics $s_1s_4$, $s_2^3$ and $s_6$  
 belong to the pencil of cubics which pass through $(1:0:0)$ (resp $(0:0:1)$) 
 and are osculating to the same cubic, $x_0^2x_1+x_1^2x_2+x_2^3=0$, with multiplicity 3 (resp. 6) in this point.  
 Thus there exists a linear relation among them, which corresponds to the second generator of $J$. 
 Finally, the three cubics $s_1s_3^2$, $s_2^3$ and $s_5$  
 belong to the pencil of cubics with a cusp in $(0:1:0)$ whose principal tangent is $x_0=0$ 
 and which are osculating with multiplicity 3 to the cubic $x_0^2x_1+x_2^2$ in $(1:0:0)$.  
 The linear relation between them corresponds to the third generator of $J$. 
 \end{remark}
 
The procedure in the previous example 
allows to compute 
the ideal of relations of the Cox ring 
of all the extremal rational elliptic surfaces,
except for the surfaces $X_{8211}, X_{9111}$ 
and $X_{4422}$, where we are not able to identify the 
ideal $J'$ for computational reasons. 

\begin{theorem}
The Cox rings of all extremal rational elliptic surfaces 
are given in Table \ref{cox}.
\end{theorem}
\begin{proof}
In case the surface has complexity one, its Cox ring is computed 
in section \ref{c*}.
Otherwise, the Cox rings are computed by means of Theorem \ref{teo}
and Corollary \ref{J}.
In order to apply Corollary \ref{J}, we recall that here $Y=\pp^2$ 
and the birational morphism  $f:X\to Y$ is chosen to be 
the blow up of the base locus of one of the pencils of cubics given in 
Table \ref{equa}. The sections $s_1,\dots,s_m\in \cc[x_0,x_1,x_2]$ 
defining the homomorphism $\phi$ in Corollary \ref{J} are the following:
%for the extremal rational elliptic surfaces which have no $\cc^*$ action and 
%which are distict from $X_{8211}, X_{9111}$ 
%and $X_{4422}$:
\[
\begin{array}{c|l}
%X_{22} & x_3,\ x_2,\ x_1,\ x_2^3+x_1^2x_3\\
%X_{211} & x_3,\ x_2,\ x_1,\ x_2^3+x_1^2x_3+x_2^2x_3\\
X_{411} & x_1,\ x_2,\ x_0,\ x_0^2+x_1x_2,\ x_0^2x_1+x_2^3,\ x_0^2x_1+x_2^3+x_1^2x_2\\[3pt]
%X_{33} & x_3,\ x_1,\ x_1x_3-x_2^2,\ x_2\\
%X_{321} & x_3,\ x_1,\ x_1x_3-x_2^2+x_2x_3,\ x_2\\
%X_{44} & x_1,\ x_2,\ x_3,\ x_2-x_3\\
%X_{431} & x_1,\ x_2,\ x_3,\ x_1+x_2+x_3\\
X_{222} & x_0-x_1,\ x_0,\ x_1,\ x_1-x_2,\ x_1^3+2x_0x_1x_2-2x_1^2x_2-x_0x_2^2+x_1x_2^2,\\ [3pt]
	& x_1^3+2x_0x_1x_2-2x_1^2x_2-x_0x_2^2+x_1x_2^2+x_0x_1(x_0-x_1)\\[3pt]
X_{141} & x_0,\ x_1,\ x_0-x_1,\ x_2,\ x_0x_1-x_1^2+x_0x_2\\[3pt]
X_{6321} & x_0,\ x_1,\ x_2,\ x_0+x_2,\ x_0+x_1,\ x_1+x_2,\ x_0+x_1+x_2,\ x_0x_1+x_0x_2+x_1x_2 \\[3pt]
X_{5511} & x_0,\ x_2,\ x_1,\ x_0+x_1,\ x_0+x_1+x_2,\ x_1+x_2\\[3pt]
X_{3333} & x_2,\ x_1,\ x_0,\ x_0 + x_1 + x_2,\ x_0 + x_1 + \epsilon^2 x_2,\\
 &   x_0 + x_1 + \epsilon x_2,\ x_0 + \epsilon^2x_1 + x_2,\ x_0 + \epsilon^2x_1 + \epsilon^2x_2,\\
   & x_0 + \epsilon ^2x_1 + \epsilon x_2,\ x_0 + \epsilon x_1 + x_2,\ x_0 + \epsilon x_1 + \epsilon^2x_2,\ x_0 + \epsilon x_1 + \epsilon x_2
\end{array}
\]
where $\epsilon$ is a primitive cube root of unity.
\end{proof}
\begin{remark}
The ideal of relation of the surface $X_{3333}$ is given up to the action 
of its automorphism group $G$, the Hessian group \cite{AD}.
\end{remark}

\section{Applications}
Let $X$ be an extremal rational elliptic surface. Any contraction of $(-1)$-curves 
and $(-2)$-curves gives a birational morphism 
 $X\to X'$, where $X'$ is still a Mori dream space by \cite[Theorem 1.1]{O}.
Moreover, by \cite[Theorem 1.3.3]{ADHL}, the morphism $X\to X'$ 
comes from a toric ambient modification,
so that the Cox ring $\R(X')$ is obtained from $\R(X)$ by putting $T_i=1$ 
for any generator $T_i$ of $\R(X)$ 
which defines an exceptional divisor of the morphism.
This allows one to compute the Cox ring of several generalized 
(and possibly singular) del Pezzo surfaces.
We will give only two examples here.
\begin{example}
The Cayley's cubic surface $Y$ is the 
only cubic surface in $\pp^3$ 
with four ordinary double points:
\[
Y:\ \frac{1}{x_0}+\frac{1}{x_1}+\frac{1}{x_2}+\frac{1}{x_3}=0.
\]
The surface is constructed as follows.
Let $l_1,l_2,l_3,l_4$ be the defining
polynomials of four general lines of 
$\pp^2$ whose unique linear relation
is $l_1+l_2+l_3+l_4=0$. Then
the Cayley cubic $Y$ is the image of the
rational map $\pp^2\to\pp^3$ defined
by
\[
 (x_0:x_1:x_2)\mapsto(l_1l_2l_3:l_1l_2l_4:l_1l_3l_4:l_2l_3l_4).
\]
Indeed the products of three of the
above reducible cubics equal $l_i(l_1l_2l_3l_4)^2$,
for $i\in\{1,2,3,4\}$, and thus their sum
is zero by construction.
Hence the minimal resolution $\tilde{Y}$ of $Y$
is the blow-up of $\pp^2$ along the six 
points of intersection of the four lines.
In particular $\tilde{Y}$ contains four
$(-2)$-curves, corresponding to the four 
lines, and nine $(-1)$-curves corresponding
to the six intersection points plus the
strict transforms of three lines through
pairs of these points.

Contracting the three $(-1)$-curves of
$X_{6321}$ corresponding to the generators
$T_{10}$, $T_{12}$, $T_{14}$, with the notation 
in Table \ref{cox}, it is easy
to verify that one obtains a smooth surface
whose configuration of negative curves
is exactly the one of $\tilde{Y}$.
Hence the Cox ring of $\tilde{Y}$ is obtained by
that of $X_{6321}$ by putting equal to 
one the variables $T_{10}, T_{12}, T_{14}$.
In this case $\R(\tilde Y)$ has $13$ generators, which 
correspond to the negative curves, and the ideal 
of relations is
\[
I(\tilde Y)=\left\langle
\begin{array}{c}
    T_1T_{11}T_{15} + T_{3}T_{13}T_{17} - T_{4}T_{16},
   T_1T_9T_{15} + T_2T_{13}T_{16} - T_5T_{17},\\
   T_{2}T_{11}T_{16} + T_{3}T_{9}T_{17} - T_{6}T_{15},
   T_{1}T_{9}T_{11} + T_6T_{13} - T_7T_{16}T_{17},\\
   T_{3}T_{9}T_{13} + T_5T_{11} - T_7T_{15}T_{16},
   T_2T_{11}T_{13} + T_4T_9 - T_7T_{15}T_{17},\\
   T_1T_3T_9^2 + T_2T_7T_{16}^2 - T_5T_6,
   T_1T_2T_{11}^2 + T_3T_7T_{17}^2 - T_4T_6,\\
   T_1T_7T_{15}^2 + T_2T_3T_{13}^2 - T_{4}T_5
   \end{array}
\right\rangle
.\] 
The relations above are induced by the pencils of lines through $p_1,\dots,p_6$ and 
by the pencils of conics through four of the points (such that no three of them are collinear).
In fact each such pencil clearly contains three sections which are polynomials in the generators of the Cox ring.

To obtain $Y$ it is enough to further
contract the $(-2)$-curves corresponding
to $T_1,T_2,T_3, T_7$. 
%$P_1,3P_1,5P_1$, where $P_1$ is 
%a generator of the Mordell-Weil group, and the four $(-2)$-curves which 
%do not meet them.
%the three $(-1)$-curves are defined by %Thus the Cox ring of $Y$ is generated by 
%$T_4,T_5,T_6,T_8,T_9,T_{11},T_{13}, T_{15}, T_{16}, T_{17}$.
When putting equal to one the previous variables in the generators of $I(X)$, 
one equation contains a pure linear term in the variable $T_8$.
After eliminating the variable $T_8$, we obtain that 
$\R(Y)$  is generated by $T_4$,$T_5$,$T_6$,$T_9$,
$T_{11}$,$T_{13}$, $T_{15}$, $T_{16}$, $T_{17}$
and its ideal of relations is
\[ 
I(Y)=\left\langle 
\begin{array}{c}
 T_6T_{15} - T_9T_{17} - T_{11}T_{16},\
 T_6T_{13} + T_9T_{11} - T_{16}T_{17},\\
 T_4T_5 - T_{13}^2 - T_{15}^2,\  T_4T_9 + T_{11}T_{13} - T_{15}T_{17}\\ 
  T_{5}T_{17} - T_{9}T_{15} - T_{13}T_{16},\
   T_{5}T_{11} + T_{9}T_{13} - T_{15}T_{16},\\ 
    T_9T_{11}T_{15}T_{16} + T_9T_{13}T_{16}T_{17} + T_{11}T_{13}T_{16}^2 - 
        T_{15}T_{16}^2T_{17},\\
    T_5T_6 - T_9^2 - T_{16}^2,\ T_4T_6 - T_{11}^2 - T_{17}^2,\
    T_4T_{16} - T_{11}T_{15} - T_{13}T_{17}
 \end{array}
\right\rangle.
\]
Observe that the degrees of the generators of $\R(Y)$  
are given by the classes of $\deg(T_i)$ in the quotient 
$\Cl(X)/K\cong \zz^3\oplus \zz/2\zz$, where $K$ is the subgroup 
generated by the degrees of the variables which have been put equal to one.

%The minimal resolution $\tilde Y$ of $Y$ is the blow-up of the projective 
%plane at the six intersection points $p_1,\dots,p_6$ of four general lines.
\end{example}

\begin{example}
In \cite{De} Derenthal classified the generalized del Pezzo 
surfaces whose Cox ring has only one relation and described its Cox rings.
An easy check shows that, by contracting $(-1)$-curves on extremal rational elliptic surfaces, 
a surface for each of the types in \cite[Theorem 2]{De} can be obtained. 
This allows to give an alternative computation of their Cox ring in several cases. 
%all cases except for the del Pezzo of type $2A_3+A_1$ and $2A_2+A_1$, 
%which are dominated by the surface $X_{4422}$.
For example, it is known that there are only two del Pezzo surfaces $Y_0,Y_1$ 
of degree one and type $E_8$ up to isomorphism \cite{AN}.
These can be obtained by contracting the unique section in the surface $X_{22}$ 
and in the surface $X_{211}$ respectively.
Their Cox rings are generated by $T_1,\dots,T_{12}$ and the ideal of relations is:
\[
I(Y_\lambda)=\langle T_1T_3^2+T_2^3T_5^2T_6+T_4T_8T_9^2T_{10}^3T_{11}^4T_{12}^5+
\lambda T_1T_2^2T_5^2T_6^2T_7^2T_8^2T_9^2T_{10}^2T_{11}^2T_{12}^2\rangle
,\]
with $\lambda=0,1$.
\end{example}

\newpage
\section{The Cox rings}
\begin{table}[!h]\label{coxrings1}
%\begin{center}
\begin{tabular}{l|c}
Surface & Degree matrix and $I(X)$\\
\midrule
\multirow{2}{*}{$X_{22}$} & ${\tiny \left(
\begin{tabular}{cccc|ccccccccc} 
1 & 1 & 1& 3 & 0 & 0 & 0 & 0 & 0 & 0 & 0 & 0 & 0\\
-1& -1 & 0 & -1 &1 & 0 & 0 & 0 & 0 & 0 & 0 & 0 & 0\\
-1 & 0 & 0 & -1 &  -1& 1 & 0 & 0 & 0 & 0 & 0 & 0 & 0\\
-1 & 0 & 0 & -1 & 0 & -1 & 1 & 0 & 0 & 0 & 0 & 0 & 0\\
0 & 0 & 0 & -1  & 0 & 0  & -1 & 1& 0 & 0 & 0 & 0 & 0\\
0 & 0 & 0 & -1 & 0 & 0 & 0 & -1 & 1 & 0 & 0 & 0 & 0 \\
0 & 0 & 0 & -1 & 0 & 0 & 0 & 0 & -1 & 1 & 0 & 0 & 0  \\
0 & 0 & 0 & -1  & 0 & 0 & 0 & 0 & 0 & -1 & 1 & 0 & 0\\
0 & 0 & 0 & -1  & 0 & 0 & 0 & 0 & 0 & 0 & -1 & 1 & 0\\
0 & 0 & 0 & -1 & 0 & 0 & 0 & 0 & 0 & 0 & 0 & -1 & 1
\end{tabular}\right)}$ \\[40pt]
& ${\tiny
\begin{array}{l}
I(X)=\langle T_{1}T_{3}^2 + T_{2}^3T_{5}^2T_{6} - T_{4}T_{8}T_{9}^2T_{10}^3T_{11}^4T_{12}^5T_{13}^6\rangle
\end{array}}
$
\\
\midrule
 \multirow{2}{*}{$X_{211}$} & ${\tiny \left(
\begin{tabular}{cccc|ccccccccc} 
1 & 1 & 1& 3 & 0 & 0 & 0 & 0 & 0 & 0 & 0 & 0 & 0\\
-1& -1 & 0 & -1 &1 & 0 & 0 & 0 & 0 & 0 & 0 & 0 & 0\\
-1 & 0 & 0 & -1 &  -1& 1 & 0 & 0 & 0 & 0 & 0 & 0 & 0\\
-1 & 0 & 0 & -1 & 0 & -1 & 1 & 0 & 0 & 0 & 0 & 0 & 0\\
0 & 0 & 0 & -1  & 0 & 0  & -1 & 1& 0 & 0 & 0 & 0 & 0\\
0 & 0 & 0 & -1 & 0 & 0 & 0 & -1 & 1 & 0 & 0 & 0 & 0 \\
0 & 0 & 0 & -1 & 0 & 0 & 0 & 0 & -1 & 1 & 0 & 0 & 0  \\
0 & 0 & 0 & -1  & 0 & 0 & 0 & 0 & 0 & -1 & 1 & 0 & 0\\
0 & 0 & 0 & -1  & 0 & 0 & 0 & 0 & 0 & 0 & -1 & 1 & 0\\
0 & 0 & 0 & -1 & 0 & 0 & 0 & 0 & 0 & 0 & 0 & -1 & 1
\end{tabular}\right)}$ \\[40pt]
& ${\tiny
\begin{array}{l}
I(X)=\langle T_{1}T_{3}^2 + T_{2}^3T_{5}^2T_{6} - T_{4}T_{8}T_{9}^2T_{10}^3T_{11}^4T_{12}^5T_{13}^6+
T_{1}T_{2}^2T_{5}^2T_{6}^2T_{7}^2T_{8}^2T_{9}^2T_{10}^2T_{11}^2T_{12}^2T_{13}^2\rangle
\end{array}}
$
\\

\midrule

\multirow{2}{*}{$X_{411}$} & ${\tiny \left(
\begin{tabular}{cccccc|ccccccccc} 
1  & 1   & 1  & 2& 3& 3 & 0 & 0 & 0 & 0 & 0 & 0 & 0 & 0 & 0\\  
-1 & -1 & 0 & 0 & -1 & -1 & 0 & 0 & 0 & 0 & 0 & 0 & 0 & 0 & 1\\  
-1 & 0 & 0 & 0 & -1 & -1 & 0 & 0 & 0 & 0 & 0 & 0 & 0 & 1 & -1\\   
-1 &  0 &  0 & 0 & -1 &  -1& 0 & 0 & 0 & 0 & 0 & 0 & 1 & -1 & 0\\
0  & 0 &  0 &  0 & -1 & -1 & 0 & 0 & 0 & 0 & 0 & 1 & -1 & 0 & 0\\
 0 & 0 & 0 & 0 &-1 &-1 & 0 & 0 & 0 & 0 & 1 & -1 & 0 & 0 & 0\\
 0 &-1 &-1 &-1& -2& -1 & 0 & 0 & 0 & 1 & 0 & 0 & 0 & 0 & 0\\
 0 &-1  &0 &-1  &0 &-1 & 0 & 0 & 1 & -1 & 0 & 0 & 0 & 0 & 0\\
 0 & 0 & 0 & -1 & 0 & -1& 0 & 1 & -1 & 0 & 0 & 0 & 0 & 0 & 0\\
 0  &0 &  0  &-1 &  0  &-1& 1 & -1 & 0 & 0 & 0 & 0 & 0 & 0 & 0
\end{tabular}\right)}$ \\[40pt]
& {\tiny $I(X)=\left\langle
\begin{array}{c}
T_{7}^2T_{8}^3T_{9}^4T_{10}T_{11}^2T_{12}^3T_{13}^4T_{14}^4T_{15}^4T_{1}T_{2}^4  - T_{4}T_{5}+T_{3}^2T_{6},\\[2pt]
T_{7}^2T_{8}^3T_{9}^4T_{10}^2T_{14}T_{15}^2T_{2}^3 - T_{10}^2T_{11}T_{6} + T_{1}T_{4},\ 
T_{11}^2T_{12}^3T_{13}^4T_{14}^3T_{15}^2T_{1}^2T_{2}-T_{7}^2T_{8}T_{6} + T_{10}T_{5},\\[2pt]
 T_{7}^4T_{8}^4T_{9}^4T_{10}T_{14}T_{15}^2T_{2}^3 - T_{11}^2T_{12}T_{5} + T_{1}T_{3}^2,\ 
T_{11}^4T_{12}^4T_{13}^4T_{14}^3T_{15}^2T_{1}T_{2}	-T_{7}^2T_{8}T_{4} + T_{10}T_{3}^2\\ [2pt]
\end{array}
\right\rangle
$}
\\
\midrule

\multirow{2}{*}{$X_{9111}$} & ${\tiny \left(
\begin{tabular}{cccccccccc|ccccccccc} 
    1&1&1&2&2&2&3&4&4&4&0&0&0&0&0&0&0&0&0\\
  -1&0 &-1&0 &-1 &-1 &-1&0 &-1 &-2&1&0&0&0&0&0&0&0&0\\
   0&0 &-1&0&0 &-1 &-1&0 &-1 &-2 &-1&1&0&0&0&0&0&0&0\\
   0&0&0&0&0 &-1 &-1&0 &-1 &-2&0 &-1&1&0&0&0&0&0&0\\
  -1 &-1&0 &-1&0 &-1 &-1 &-2&0 &-1&0&0&0&1&0&0&0&0&0\\
  -1&0&0 &-1&0&0 &-1 &-2&0 &-1&0&0&0 &-1&1&0&0&0&0\\
   0&0&0 &-1&0&0 &-1 &-2&0 &-1&0&0&0&0 &-1&1&0&0&0\\
   0 &-1 &-1 &-1 &-1&0 &-1 &-1 &-2&0&0&0&0&0&0&0&1&0&0\\
   0 &-1&0&0 &-1&0 &-1 &-1 &-2&0&0&0&0&0&0&0 &-1&1&0\\
   0&0&0&0 &-1&0 &-1 &-1 &-2&0&0&0&0&0&0&0&0 &-1&1
\end{tabular}\right)}$ \\[40pt]
\midrule
\multirow{2}{*}{$X_{33}$} & ${\tiny \left(
\begin{tabular}{cccc|ccccccccc} 
1 & 1 & 2 & 1 & 0 & 0 & 0 & 0 & 0 & 0 & 0 & 0 & 0\\
-1 & 0 & -1 & -1 & 1 & 0 & 0 & 0 & 0 & 0 & 0 & 0 & 0 \\
-1 & 0 & -1 & 0 & -1 & 1 & 0 & 0 & 0 & 0 & 0 & 0 &0 \\
0 & 0 & -1 & 0 & 0 & -1 & 1 & 0 & 0 & 0 & 0 & 0 & 0\\
0 & 0 & -1 & 0 & 0 & 0 & -1 & 1 & 0 & 0 & 0 & 0 & 0\\
0 & 0 & -1 & 0 & 0 & 0 & 0 & -1 & 1 & 0 & 0 & 0 & 0\\
0 & 0 & -1 & 0 & 0 & 0 & 0 & 0 & -1 & 1 & 0 & 0 & 0\\
-1 & -1& 0 & 0 & 0 & 0 & 0 & 0 & 0 & 0 & 1 & 0 & 0\\
0 & -1 & 0 & 0 & 0 & 0 & 0 & 0 & 0 & 0 & -1 & 1 & 0\\
0 & -1 & 0 & 0 & 0 & 0 & 0 & 0 & 0 & 0 & 0 & -1 & 1\\
\end{tabular}\right)}$ \\[40pt]
& ${\tiny
\begin{array}{l}
I(X)=\langle  T_{1}T_{2}T_{11}^2T_{12}^3T_{13}^4 -T_{3}T_{7}T_{8}^2T_{9}^3T_{10}^4 -T_{4}^2T_{5} \rangle
\end{array}}
$
\end{tabular}
\end{table}

\begin{table}[h]\label{coxrings2}
\begin{tabular}{l|c}
Surface & Degree matrix and $I(X)$\\
\midrule
\multirow{2}{*}{$X_{321}$} & ${\tiny \left(
\begin{tabular}{cccc|ccccccccc} 
1 & 1 & 2 & 1 & 0 & 0 & 0 & 0 & 0 & 0 & 0 & 0 & 0\\
-1 & 0 & -1 & -1 & 1 & 0 & 0 & 0 & 0 & 0 & 0 & 0 & 0 \\
-1 & 0 & -1 & 0 & -1 & 1 & 0 & 0 & 0 & 0 & 0 & 0 &0 \\
0 & 0 & -1 & 0 & 0 & -1 & 1 & 0 & 0 & 0 & 0 & 0 & 0\\
0 & 0 & -1 & 0 & 0 & 0 & -1 & 1 & 0 & 0 & 0 & 0 & 0\\
0 & 0 & -1 & 0 & 0 & 0 & 0 & -1 & 1 & 0 & 0 & 0 & 0\\
0 & 0 & -1 & 0 & 0 & 0 & 0 & 0 & -1 & 1 & 0 & 0 & 0\\
-1 & -1& 0 & 0 & 0 & 0 & 0 & 0 & 0 & 0 & 1 & 0 & 0\\
0 & -1 & 0 & 0 & 0 & 0 & 0 & 0 & 0 & 0 & -1 & 1 & 0\\
0 & -1 & 0 & 0 & 0 & 0 & 0 & 0 & 0 & 0 & 0 & -1 & 1\\
\end{tabular}\right)}$ \\[40pt]
& ${\tiny
\begin{array}{l}
 I(X)=\langle T_{1}T_{2}T_{11}^2T_{12}^3T_{13}^4 -T_{3}T_{7}T_{8}^2T_{9}^3T_{10}^4 -T_{4}^2T_{5}+
  T_{1}T_{4}T_{5}T_{6}T_{7}T_{8}T_{9}T_{10}T_{11}T_{12}T_{13}\rangle
  \end{array}}
$
\\
\midrule

\multirow{2}{*}{$X_{8211}$} & ${\tiny \left(
\begin{tabular}{ccccccccc|ccccccccc} 
 1& 1 & 1 & 1  &2&2&2&3&3&0&0&0&0&0&0&0&0&0\\
-1&0 &-1 &-1&0&0&0 &-1 & -1&1&0&0&0&0&0&0&0&0\\
 0&0&0 & -1&0&0&0 &-1 &-1 &-1&1&0&0&0&0&0&0&0\\
-1 &-1&0&0 &-1&0 & -1 & -2 & -1&0&0&1&0&0&0&0&0&0\\
-1&0&0&0 &-1&0 &-1&0 &-1&0&0 &-1&1&0&0&0&0&0\\
 0&0&0&0 &-1&0 &-1&0 &-1&0&0&0 &-1&1&0&0&0&0\\
 0 &-1 &-1&0 &-1 &-1&0 &-1 &-2&0&0&0&0&0&1&0&0&0\\
 0&0 & -1&0 & -1 & -1&0 & -1&0&0&0&0&0&0 &-1&1&0&0\\
 0&0&0&0 &-1 & -1&0 &-1&0&0&0&0&0&0&0 &-1&1&0\\
 0 &-1&0 & -1&0 & -1 & -1&0&0&0&0&0&0&0&0&0&0&1\\
\end{tabular}\right)}$ \\[40pt]
 \\
\midrule

\multirow{2}{*}{$X_{44}$} & ${\tiny \left(
\begin{tabular}{cccc|ccccccccc} 
 1 & 1 & 1& 1 & 0 & 0 & 0 & 0 & 0 & 0 & 0 & 0 & 0\\
 -1 & -1 & 0 & 0 & 1 & 0 & 0 & 0 & 0 & 0 & 0 & 0 & 0\\
 0 & -1 & 0 & 0 & -1 & 1 & 0 & 0 & 0 & 0 & 0 & 0 & 0\\
 0 & -1 & 0 & 0 &0 & -1 & 1 & 0 & 0 & 0 & 0 & 0 & 0\\
 -1 & 0 & -1 & 0 & 0 & 0 & 0 & 1 & 0 & 0 & 0 & 0 & 0\\
 0 & 0 & -1 & 0 & 0 & 0 & 0 & -1 & 1 & 0 & 0 & 0 & 0\\
 0 & 0 & -1 & 0 & 0 & 0 & 0 & 0 & -1 & 1 & 0 & 0 & 0\\
 -1 & 0 & 0 & -1 & 0 & 0 & 0 & 0 & 0 & 0 & 1 & 0 & 0\\
 0 & 0 & 0 & -1 & 0 & 0 & 0 & 0 & 0 & 0 & -1 & 1 & 0\\
 0 & 0 & 0 & -1 & 0 & 0 & 0 & 0 & 0 & 0 & 0 & -1 & 1\\ 
\end{tabular}\right)}$ \\[40pt]
& ${\tiny
\begin{array}{l}
 I(X)=\langle T_{2}T_{5}T_{6}^2T_{7}^3 - T_{3}T_{8}T_{9}^2T_{10}^3 - T_{4}T_{11}T_{12}^2T_{13}^3\rangle
  \end{array}}
$
\\
%\midrule

\midrule
\multirow{2}{*}{$X_{431}$} & ${\tiny \left(
\begin{tabular}{cccc|ccccccccc} 
 1 & 1 & 1& 1 & 0 & 0 & 0 & 0 & 0 & 0 & 0 & 0 & 0\\
 -1 & -1 & 0 & 0 & 1 & 0 & 0 & 0 & 0 & 0 & 0 & 0 & 0\\
 0 & -1 & 0 & 0 & -1 & 1 & 0 & 0 & 0 & 0 & 0 & 0 & 0\\
 0 & -1 & 0 & 0 &0 & -1 & 1 & 0 & 0 & 0 & 0 & 0 & 0\\
 -1 & 0 & -1 & 0 & 0 & 0 & 0 & 1 & 0 & 0 & 0 & 0 & 0\\
 0 & 0 & -1 & 0 & 0 & 0 & 0 & -1 & 1 & 0 & 0 & 0 & 0\\
 0 & 0 & -1 & 0 & 0 & 0 & 0 & 0 & -1 & 1 & 0 & 0 & 0\\
 -1 & 0 & 0 & -1 & 0 & 0 & 0 & 0 & 0 & 0 & 1 & 0 & 0\\
 0 & 0 & 0 & -1 & 0 & 0 & 0 & 0 & 0 & 0 & -1 & 1 & 0\\
 0 & 0 & 0 & -1 & 0 & 0 & 0 & 0 & 0 & 0 & 0 & -1 & 1\\ 
\end{tabular}\right)}$ \\[40pt]
& ${\tiny
\begin{array}{l}
I(X)=\langle  T_{2}T_{5}T_{6}^2T_{7}^3 + T_{3}T_{8}T_{9}^2T_{10}^3 - T_{4}T_{11}T_{12}^2T_{13}^3+ 
  T_{1}T_{5}T_{6}T_{7}T_{8}T_{9}T_{10}T_{11}T_{12}T_{13} \rangle
   \end{array}}
$\\

\midrule

\multirow{2}{*}{$X_{222}$} & ${\tiny \left(
\begin{tabular}{cccccc|ccccccccc} 
 1 & 1 & 1 & 1 & 3 & 3 & 0 & 0 & 0 & 0 & 0 & 0 & 0 & 0 & 0\\
 -1 & -1 & -1 & 0 & -1 & -1 &0 & 0 & 0 & 0 & 1 & 0 & 0 & 0 & 0\\
 -1 & 0 & 0 & 0 & -1 & -1 &0 & 0 & 0 & 1 & -1 & 0 & 0 & 0 & 0\\
 -1 & 0 & 0 & 0 & -1 & -1 &0 & 0 & 1 & -1 & 0 & 0 & 0 & 0 & 0\\
 0 & 0 & 0 & 0 & -1 & -1 & 0 & 1 & -1 & 0 & 0 & 0 & 0 & 0 & 0\\
 0 & 0 & 0 & 0 & -1 & -1 & 1 & -1 & 0 & 0 & 0 & 0 & 0 & 0 & 0\\
 0 & 0 & -1 & -1 & -2 & -1 &0 & 0 & 0 & 0 & 0 & 0 & 1 & 0 & 0\\
 0 & 0 & -1 & 0 & 0 & -1 & 0 & 0 & 0 & 0 & 0 & 1 & -1 & 0 & 0\\
 0 & -1 & 0 & -1 & -1 & -2 & 0 & 0 & 0 & 0 & 0 & 0 & 0 & 0 & 1\\
 0 & -1 & 0 & 0 & -1 & 0 & 0 & 0 & 0 & 0 & 0 & 0 & 0 & 1 & -1\\
 \end{tabular}\right)}$ \\[40pt]
& {\tiny $I(X)=\left\langle
\begin{array}{c}
 T_{1}T_{7}^2T_{8}^2T_{9}^2T_{10} - T_{2}T_{14}^2T_{15} + T_{3}T_{12}^2T_{13},\ 
 T_{1}^2T_{4}^2T_{8}T_{9}^2T_{10} + T_{2}T_{5}T_{14}^2 - T_{3}T_{6}T_{12}^2,\\[2pt]
 T_{1}T_{4}^2T_{13} - T_{2}^2T_{3}T_{10}T_{11}^2T_{14}^2 + T_{6}T_{7}^2T_{8},\ 
 T_{1}T_{4}^2T_{15} - T_{2}T_{3}^2T_{10}T_{11}^2T_{12}^2 + T_{5}T_{7}^2T_{8},\\[2pt]
 T_{1}T_{2}T_{3}T_{8}T_{9}^2T_{10}^2T_{11}^2 + T_{5}T_{13} - T_{6}T_{15}\\[2pt]
 
\end{array}
\right\rangle
$}
\\
\midrule

\multirow{2}{*}{$X_{141}$} & ${\tiny \left(
\begin{tabular}{ccccc|cccccccccc} 
 1 & 1 & 1 & 1 & 2 & 0 & 0 & 0 & 0 & 0 & 0 & 0 & 0 & 0\\
 -1 & -1 & -1 & 0 & -1 &0 & 0 & 0 & 1 & 0 & 0 & 0 & 0 & 0\\
 -1 & 0 & 0 & 0 & -1 & 0 & 0 & 1 & -1 & 0 & 0 & 0 & 0 & 0\\
 0 & 0 & 0 & 0 & -1 & 0 & 1& -1 & 0 & 0 & 0 & 0 & 0 & 0\\ 
 0 & 0 & 0 & 0 & -1 & 1 & -1 & 0 & 0 & 0 & 0 & 0 & 0 & 0\\ 
 0 & -1 & 0 & -1 & -1 & 0 & 0 & 0 & 0 & 1 & 0 & 0 & 0 & 0\\
 0 & -1 & 0 & 0 & 0 & 0 & 0 & 0 & 0 & -1 & 1 & 0 & 0 & 0\\
 0 & 0 & -1 & -1 & -1 & 0 & 0 & 0 & 0 & 0 & 0 & 1 & 0 & 0\\
 0 & 0 & -1 & 0 & 0 & 0 & 0 & 0 & 0 & 0 & 0 & -1 & 1 & 0\\
 -1 & 0 & 0 & -1 & 0 & 0 & 0 & 0 & 0 & 0 & 0 & 0 & 0 & 1
 \end{tabular}\right)}$ \\[40pt]
& {\tiny $I(X)=\left\langle
\begin{array}{l}
T_{1}T_{4}T_{14}^2 + T_{2}T_{3}T_{9}T_{11}T_{13} - T_{5}T_{6}^2T_{7},\ 
T_{1}T_{6}T_{7}T_{8}T_{14} - T_{2}T_{10}T_{11}^2 - T_{3}T_{12}T_{13}^2
 
\end{array}
\right\rangle
$}

  \end{tabular}
%\end{center}

\end{table}

  \begin{table}[!h]\label{coxrings3}
\begin{tabular}{l|c}
Surface & Degree matrix and $I(X)$\\
  
\midrule
\multirow{2}{*}{$X_{6321}$} & ${\tiny \left(
\begin{tabular}{cccccccc|ccccccccc} 
1 & 1 & 1 & 1 & 1 & 1 & 1 & 2 & 0 & 0 & 0 & 0 & 0 & 0 & 0 & 0 & 0\\
-1 & 0 & -1 & -1 & 0 & 0 & 0 & -1 &1 & 0 & 0 & 0 & 0 & 0 & 0 & 0 & 0\\
0 & 0 & 0 & -1 & 0 & 0 & 0 & -1 & -1 & 1 & 0 & 0 & 0 & 0 & 0 & 0 & 0\\
-1& -1& 0 & 0 & -1 & 0 & 0 & -1 & 0 & 0 & 1 & 0 & 0 & 0 & 0 & 0 & 0\\
0 & 0 & 0 & 0 & -1 & 0 & 0 & -1 & 0 & 0 & -1 & 1 & 0 & 0 & 0 & 0 & 0\\
0 & -1 & -1 & 0 & 0 & -1 & 0 & -1 & 0 & 0 & 0 & 0 & 1 & 0 & 0 & 0 & 0\\
0 & 0 & 0 & 0 & 0 & -1 & 0 & -1 & 0 & 0 & 0 & 0 & -1 & 1 & 0 & 0 & 0\\
-1 & 0 & 0 & 0 & 0 & -1 & -1 & 0 & 0 & 0 & 0 & 0 & 0 & 0 & 1 & 0 & 0\\
0 & -1 & 0 & -1 & 0 & 0 & -1 & 0 & 0 & 0 & 0 & 0 & 0 & 0 & 0 & 1 & 0\\
0 & 0 & -1 & 0 &-1 & 0  &-1 & 0 & 0 & 0 & 0 & 0 & 0 & 0 & 0 & 0 & 1
   \end{tabular}\right)}$ \\[40pt]
& {\tiny $I(X)=\left\langle
\begin{array}{c}
T_{2}T_{11}T_{12}T_{16} + T_{3}T_{9}T_{10}T_{17} - T_{6}T_{14}T_{15},\
T_{1}T_{11}T_{12}T_{15} + T_{3}T_{13}T_{14}T_{17} - T_{4}T_{10}T_{16},\\[2pt]
T_{1}T_{9}T_{10}T_{15} + T_{2}T_{13}T_{14}T_{16} - T_{5}T_{12}T_{17},\
T_{1}T_{3}T_{9}T_{15}T_{17} + T_{2}T_{4}T_{16}^2 - T_{8}T_{12}T_{14},\\[2pt]
T_{1}T_{2}T_{11}T_{15}T_{16} + T_{3}T_{5}T_{17}^2 - T_{8}T_{10}T_{14},\
T_{3}T_{9}T_{10}T_{13}T_{14} + T_{5}T_{11}T_{12}^2 - T_{7}T_{15}T_{16},\\[2pt]
T_{1}T_{9}T_{10}T_{11}T_{12} + T_{6}T_{13}T{14}^2 - T_{7}T_{16}T_{17},\
T_{1}T_{6}T_{15}^2 + T_{2}T_{3}T_{13}T_{16}T_{17} - T_{8}T_{10}T_{12},\\[2pt]
T_{2}T_{11}T_{12}T_{13}T_{14} + T_{4}T_{9}T_{10}^2 - T_{7}T_{15}T_{17},\
T_{1}^2T_{9}T_{11}T_{15}^2 - T_{4}T_{5}T_{16}T_{17} + T_{8}T_{13}T_{14}^2,\\[2pt]
T_{1}T_{2}T_{11}^2T_{12}^2 + T_{3}T_{7}T_{17}^2 - T_{4}T_{6}T_{10}T_{14},\
T_{3}^2T_{9}T_{13}T_{17}^2 - T_{4}T_{6}T_{15}T_{16} + T_{8}T_{11}T_{12}^2,\\[2pt]
 T_{1}T_{7}T_{15}^2 + T_{2}T_{3}T_{13}^2T_{14}^2 - T_{4}T_{5}T_{10}T_{12},\
 T_{1}T_{3}T_{9}^2T_{10}^2 + T_{2}T_{7}T_{16}^2 - T_{5}T_{6}T_{12}T_{14},\\[2pt]
 T_{2}^2T_{11}T_{13}T_{16}^2 - T_{5}T_{6}T_{15}T_{17} + T_{8}T_{9}T_{10}^2,\
 T_{1}T_{2}T_{3}T_{9}T_{11}T_{13} + T_{4}T_{5}T_{6} - T_{7}T_{8}\\[2pt]

 \end{array}
\right\rangle
$}
\\
\midrule

\multirow{2}{*}{$X_{11}(a)$} & ${\tiny \left(
\begin{tabular}{ccccc|ccccccccc} 
 1 & 1 & 1 & 1 & 1 & 0 & 0 & 0 & 0 & 0 & 0 & 0 & 0 & 0\\
 -1 & 0 & 0 & -1 & 0&1 & 0 & 0 & 0 & 0 & 0 & 0 & 0 & 0 \\
 -1 & 0 & 0 & 0 & 0 &-1 & 1 & 0 & 0 & 0 & 0 & 0 & 0 & 0\\
 0 & -1 & 0 & -1 & 0 & 0 & 0 & 1 & 0 & 0 & 0 & 0 & 0 & 0\\
 0 & -1 & 0 & 0 & 0  & 0 & 0 & -1 & 1 & 0 & 0 & 0 & 0 & 0\\
 0 & 0 & -1 & -1 & 0 & 0 & 0 & 0 & 0 & 1 & 0 & 0 & 0 & 0\\
 0 & 0 & -1 & 0 & 0  & 0 & 0 & 0 & 0 & -1 & 1 & 0 & 0 & 0\\
 -1 & -1 & -1 & 0 & -1 & 0 & 0 & 0 & 0 & 0 & 0 & 1 & 0 & 0\\
 0 & 0 & 0 & 0 & -1 & 0 & 0 & 0 & 0 & 0 & 0 & -1 & 1 & 0\\
 0 & 0 & 0 & 0 & -1 & 0 & 0 & 0 & 0 & 0 & 0 & 0 & -1 & 1\\
\end{tabular}\right)}$ \\[40pt]
& {\tiny $I(X)=\left\langle
\begin{array}{l}
(a-1)T_{2}T_{8}T_{9}^2 - T_{3}T_{10}T_{11}^2 + T_{5}T_{13}T_{14}^2,\
 (a-1)T_{1}T_{6}T_{7}^2 - aT_{3}T_{10}T_{11}^2 + T_{5}T_{13}T_{14}^2
   \end{array}
\right\rangle$}
\\
\midrule

  \multirow{2}{*}{$X_{5511}$} & ${\tiny \left(
\begin{tabular}{cccccc|ccccccccccc} 
 1 & 1 & 1 & 1& 1 & 1  & 0 & 0 & 0 & 0 & 0 & 0 & 0 & 0 & 0\\
 -1 & 0 & -1 & -1 & 0 & 0 &  0 & 1 & 0 & 0 & 0 & 0 & 0 & 0 & 0\\
 0 & 0 & 0 &-1 & 0 & 0  & 1 & -1 & 0 & 0 & 0 & 0 & 0 & 0 & 0\\
 -1 & 0 & 0 & 0 & -1 & -1 &  0 & 0 & 0 & 1 & 0 & 0 & 0 & 0 & 0\\
 -1 & 0 & 0 & 0 & 0 & 0 &  0 & 0 & 1 & -1 & 0 & 0 & 0 & 0 & 0\\
 0 & -1 & 0 & -1 & -1 & 0 &  0 & 0 & 0 & 0 & 0 & 1 & 0 & 0 & 0\\
 0 & -1 & 0 & 0 & 0 & 0 &  0 & 0 & 0 & 0 & 1 & -1 & 0 & 0 & 0\\
 0 & -1 & -1 & 0 & 0 & -1 &  0 & 0 & 0 & 0 & 0 & 0 & 0 & 1 & 0\\
 0 & 0 & 0 & 0 & 0 & -1 &  0 & 0 & 0 & 0 & 0 & 0 & 1 & -1 & 0\\
 0 & 0 & -1 & 0 & -1 & 0 &  0 & 0 & 0 & 0 & 0 & 0 & 0 & 0 & 1
   \end{tabular}\right)}$ \\[40pt]
& {\tiny $I(X)=\left\langle
\begin{array}{c}
 T_{2}T_{11}T_{13}T_{14} + T_{4}T_{7}^2T_{8} - T_{5}T_{9}T_{10}T_{15},\ 
 T_{2}T_{11}^2T_{12} + T_{3}T_{7}T_{8}T_{15} - T_{6}T_{9}T_{10}T_{13},\\[2pt]
 T_{1}T_{9}^2T_{10} + T_{3}T_{13}T_{14}T_{15} - T_{4}T_{7}T_{11}T_{12},\
 T_{1}T_{7}T_{8}T_{9} - T_{5}T_{11}T_{12}T_{15} + T_{6}T_{13}^2T_{14},\\[2pt]
 T_{1}T_{2}T_{9}T_{11} + T_{3}T_{5}T_{15}^2 - T_{4}T_{6}T_{7}T_{13}\\[2pt]

 \end{array}
\right\rangle
$}
\\
\midrule

 \multirow{2}{*}{$X_{4422}$} & ${\tiny \left(
\begin{tabular}{ccccccccccc|ccccccccc} 
 1&1&1&1&1&1&1&1&2&2&1&0&0&0&0&0&0&0&0&0\\
-1&0 &-1&0&0 &-1&0&0 &-1 & -1 & -1&1&0&0&0&0&0&0&0&0\\
 0&0&0&0&0 &-1&0&0 &-1 & -1&0 &-1&1&0&0&0&0&0&0&0\\
 0 &-1&0 &-1 &-1&0&0&0 &-1 &-1 & -1&0&0&1&0&0&0&0&0&0\\
 0 &-1&0&0&0&0&0&0 &-1 &-1&0&0&0 &-1&1&0&0&0&0&0\\
-1&0&0 &-1&0&0 &-1&0&0 &-1&0&0&0&0&0&1&0&0&0&0\\
-1&0&0&0 &-1&0&0 &-1 &-1&0&0&0&0&0&0&0&1&0&0&0\\
 0 &-1&0&0&0 &-1 &-1 & -1&0&0&0&0&0&0&0&0&0&1&0&0\\
 0&0 &-1 &-1&0&0&0 &-1 &-1&0&0&0&0&0&0&0&0&0&1&0\\
 0&0 &-1&0 &-1&0 &-1&0&0 &-1&0&0&0&0&0&0&0&0&0&1\\
   \end{tabular}\right)}$ \\[40pt]
 \\
\midrule

\multirow{2}{*}{$X_{3333}$} & ${\tiny \left(
\begin{tabular}{cccccccccccc|ccccccccc} 
 1 & 1 & 1 & 1& 1 & 1 & 1 & 1 & 1 & 1 & 1 & 1 & 0 & 0 & 0 & 0 & 0 & 0 & 0 & 0 & 0 \\
 0 & 0 & -1 & -1 & 0 & 0 & 0 & -1 & 0 & 0 & 0 & -1 & 1 & 0 & 0 & 0 & 0 & 0 & 0 & 0 & 0\\
 0 & 0 & -1 & 0 & 0 &-1 & -1 & 0 & 0 & 0 & -1 & 0 & 0 & 1 & 0 & 0 & 0 & 0 & 0 & 0 & 0\\
 0 & 0 & -1 & 0 & -1 & 0 & 0 & 0 & -1 & -1 & 0 & 0 & 0 & 0 & 1 & 0 & 0 & 0 & 0 & 0 & 0\\
 0 & -1 & 0 & -1 & 0 & 0 &-1 & 0 & 0 & -1 & 0  & 0 & 0 & 0 & 0 & 1 & 0 & 0 & 0 & 0 & 0\\
 0 & -1 & 0 & 0 & 0 & -1&0 & 0 & -1 & 0 & 0 & -1 & 0 & 0 & 0 & 0 & 1 & 0 & 0 & 0 & 0\\
 0 & -1 & 0 & 0 & -1 & 0 & 0 & -1 & 0 & 0 & -1 & 0& 0 & 0 & 0 & 0 & 0 & 1 & 0 & 0 & 0\\
 -1 & 0 & 0 & -1 & -1 & -1 & 0 & 0 & 0 & 0 & 0 & 0& 0 & 0 & 0 & 0 & 0 & 0 & 1 & 0 & 0\\
 -1 & 0 & 0 & 0 & 0 & 0 & 0 & 0 & 0 & -1 & -1 & -1& 0 & 0 & 0 & 0 & 0 & 0 & 0 & 1 & 0\\
 -1 & 0 & 0 & 0 & 0 & 0 & -1 & -1 & -1 & 0 & 0 & 0& 0 & 0 & 0 & 0 & 0 & 0 & 0 & 0 & 1 
\end{tabular}\right)}$ \\[40pt]
& {\tiny $I(X)=\left\langle
\begin{array}{c}
 (-2\epsilon - 1)T_1T_{20}T_{21} - T_{5}T_{15}T_{18} + T_{6}T_{14}T_{17}\\
 (\epsilon - 1)T_1T_5T_{19}^2 + T_8T_{10}T_{13}T_{16} - T_9T_{11}T_{14}T_{17}\\
 T_1T_2T_3 + 1/9(-\epsilon + 1)T_4T_9T_{11} + 1/9(\epsilon - 1)T_5T_7T_{12}\\
\end{array}
\right\rangle^G
$}
\end{tabular}

%\hrule
\vspace{0.5cm}
\caption{The Cox rings of extremal rational elliptic surfaces}\label{cox}
\end{table}

\end{document}